\documentclass[leqno,11pt,a4paper,twoside]{article}%
\usepackage{amssymb}
\usepackage{amsfonts}
\usepackage{amsmath}
\usepackage{authblk}
\usepackage{graphicx}
\usepackage{ifthen}
\usepackage{comment}
\usepackage{xcolor}
\usepackage{braket}
\usepackage{enumitem}
\usepackage{lipsum,multicol}
\usepackage[T1]{fontenc}
\usepackage[english]{babel}
\usepackage[english]{babel}
\setlist[itemize]{noitemsep, topsep=0pt}
\setlist[enumerate]{noitemsep, topsep=0pt}
\setlist[itemize]{leftmargin=*}
\setlist[enumerate]{leftmargin=*}
\providecommand{\U}[1]{\protect\rule{.1in}{.1in}}
\providecommand{\norm}[1]{\left\lVert#1\right\rVert}%//.//
\providecommand{\abs}[1]{\left\lvert#1\right\rvert}
\providecommand{\pr}[1]{\left(#1\right)} %(.)
\providecommand{\pp}[1]{\left[#1\right]} %[.]
\providecommand{\set}[1]{\left\lbrace#1\right\rbrace} %{.}
\providecommand{\scal}[1]{\left\langle#1\right\rangle}%<.>

\newcommand{\Ho}[1]{H_0^1\pr{\mathcal{O}_{#1}}}
\newcommand{\Lo}[1]{\mathbb{L}^{#1}\pr{{\mathcal{O}}}}
\newcommand{\Hmo}{\pr{H_0^1}^{*}\pr{\mathcal{O}}}
\newcommand{\normi}[3]{\norm{#1}_{
    \ifthenelse{\equal{#2}{1}}{H_0^1\pr{\mathcal{O}_{#3}}}{%
    \ifthenelse{\equal{#2}{-1}}{H^{-1}\pr{\mathcal{O}_{#3}}}{}}}}

\makeatletter
\newcommand{\subjclass}[2][2020]{%
  \let\@oldtitle\@title%
  \gdef\@title{\@oldtitle\footnotetext{\textbf{#1 \emph{Mathematics subject classification.}} #2}}%
}
\newcommand{\keywords}[1]{%
  \let\@@oldtitle\@title%
  \gdef\@title{\@@oldtitle\footnotetext{\textbf{\emph{Key words.}} #1.}}%
}
\makeatother

\providecommand{\C}{\mathbb{C}}

\newtheorem{theorem}{Theorem}

\newtheorem{assumption}[theorem]{Assumption}

\newtheorem{definition}[theorem]{Definition}
\newtheorem{example}[theorem]{Example}

\newtheorem{problem}[theorem]{Problem}
\newtheorem{proposition}[theorem]{Proposition}
\newtheorem{remark}[theorem]{Remark}

\newenvironment{proof}[1][Proof]{\noindent\textbf{#1.} }{\ \rule{0.5em}{0.5em}}

\author[1,2]{Ioana Ciotir}
\author[3,1,4]{Dan Goreac}
\author[1,5]{Juan Li}
\author[1]{Xinru Zhang}
\affil[1]{School of Mathematics and Statistics, Shandong University, Weihai, Weihai 264209, P.R. China, xinruzhang@mail.sdu.edu.cn}
\affil[2]{Normandie University, INSA de Rouen Normandie, LMI (EA 3226 - FR CNRS 3335), 76000 Rouen, France, \textit{Email: ioana.ciotir@insa-rouen.fr}}
\affil[3]{\'{E}cole d'Actuariat, Universit\'{e} Laval,  Québec (QC), Q1V 0A6, Canada, dan.goreac@act.ulaval.ca}
\affil[4]{LAMA, Univ Gustave Eiffel, UPEM, Univ Paris Est Creteil, CNRS, F-77447 Marne-la-Vallée, France, \textit{Email: dan.goreac@univ-eiffel.fr}}
\affil[5]{Research Center for Mathematics and Interdisciplinary Sciences, Shandong University, Qingdao 266237, P. R. China, \textit{Email: juanli@sdu.edu.cn}}
\begin{document}
\title{A Stochastic Porous Media Schrödinger Equation: Feynman-type Motivation, Well-Posedness and Control Interpretation}
\maketitle

\abstract{This paper's aim is threefold. First, using Feynman's path approach to the derivation of the classical Schrödinger's equation in \cite{F_1948} and by introducing a slight path (or wave) dependency of the action, we derive a new class of equations of Schrödinger type where the driving operator is no longer the Laplace one but rather of complex porous media-type. Second, using suitable concepts of monotonicity in the complex setting and on appropriate functional spaces, we show the existence and uniqueness of the solution to this type of equation. In the formulation of our equation, we adjoin possible measurement absolute errors translating in an additive Brownian perturbation and interactions between different waves translating in a mean-field (or McKean-Vlasov) dependency of drift coefficient. Finally, using Fitzpatrick’s characterization of maximal monotone operators (cf. \cite{Fitz_1988}), we propose a Brézis-Ekeland type characterization of the solution of the deterministic equation via a control problem. This is envisaged as a possible way to overcome strict monotonicity requirements in the complex setting.}\\

\noindent{\textbf{Keywords}: equations of Schrödinger type,  Feynman's approach,  complex stochastic porous media equation,  Fitzpatrick’s characterization of complex maximal monotone operators,  control problem}\\

\noindent{MSC2020: Primary 
35J10, %Schrödinger equation
81-10, %Mathematical modeling or simulation for problems pertaining to quantum theory
76S99%Flows in porous media
; secondary 
47H05 %Monotone operators and generalizations
49J20% Existence theories for optimal control problems involving partial differential equations
\\
Classification: 40.09 Mathematical physics, 40.12 Non-linear PDE, 60.07 Quantum dynamics
}

\section{Introduction}

\noindent The mathematical developments on quantum (mechanics) systems are usually of differential nature in the spirit of Schrödinger's contributions or of algebraic nature in the spirit of Heisenberg's theory. A somewhat different approach based on Hamilton's first principle was proposed in the paper \cite{F_1948} leading to what is sometimes referred to as \emph{Feynman's formula}. The fundamental postulate in the cited reference \cite{F_1948} is the fact that, in the computation of amplitudes, the phase can be taken proportional to the \emph{"time integral of the Lagrangian along a path"} or\emph{"action"} in the nomenclature adopted. The (realization of the) path is then intended in Euclidean sense as a succession of points $\pp{x_j:=x(t_j)}_{-\infty<j<\infty}$ contributing with terms $S(x_j,x_{j+1})$ in a stationary way. In our developments, we are going to consider a simple form of Lagrangian linear in velocity ($p$) with a potential-related coefficient ($L(p,x):=\overset{\cdot }{\beta }(x)p$) but, foremost, we ask that the contribution be computed not on a straight line but on the actual wave $\psi(x,t)$, thus leading to terms like $S\pr{\psi(x_j,t_j),\psi(x_{j+1},t_{j+1}}\approx S\pr{\psi(x_j,t_j),\psi(x_{j+1},t_{j}}$. We wish to point out that, from our point of view, the underlying discrete trajectory is $\pp{(t_j,x_j:=x(t_j))}_{-\infty<j<\infty}$ with an explicit presence of the time. Furthermore, as already pointed out, $S$ is computed along the wave traveled.\\

\noindent With these considerations in mind, we derive, in Section \ref{Sec2} a Schrödinger-like PDE, albeit the fact that the governing operator is no longer the Laplace one but the $\beta$-induced porous-media nonlinear operator $\Delta \beta$, with $\beta$ being a (possibly) complex valued, complex argument function. To our best knowledge, such equations have not been previously studied and we believe that the aforementioned arguments using Feynman's approach should convince our readers of the interest of such PDEs. Let us also point out that at this level, the considerations are somewhat of an axiomatic nature and $S$ is not a priori consistent, as it depends on the wave function $\psi$. The rigorous study of $S$ once $\psi$ is obtained is left for a future work. In the case of Schrödinger's classical operator, we mention the approach in \cite[Section 5.5]{Vrabie_C0} via $C_0$-semigroups. Of course, such arguments for complex porous media operators need a careful study and, as specified, this exceeds the aim we have for this short paper.\\

\noindent In Section \ref{Sec3}, we shall adapt some monotonicity tools which
are classical in $\mathbb{R}$, to our case which needs similar results in $%
\mathbb{C}$.  The notion of monotonicity (maximal or $m$-monotonic, $m$-accretive, etc.) is not new and goes back to \cite{Kato_1967} in its complex formulation, but it is equally present, with some variations, in the original papers \cite{Minty_1962} or \cite{Browder_1968}. Furthermore, monotonic methods have been employed to treat variants of Schrödinger's equation, for instance in \cite{OY_2002} or \cite{OY_2002_Jap} in which Ginzburg-Landau and related equations are dealt with. \\

In our setting, and much like for usual porous media equations in the real setting, one needs to carefully describe the spaces on which monotonicity for the complex porous media operator can be envisaged and this constitutes the core of Section \ref{Sec3}. The choice of spaces, in the spirit of \cite{PR_2007}, is made in order to facilitate the analysis of stochastic PDEs of Schrödinger-porous media. Besides the functional definitions, we link, in Proposition \ref{PropMaxmonotone} the strict monotonicity of $\beta$ as a complex function to the maximal monotonicity of the induced porous-media operator. The equation we consider has stochastic features through a Brownian-drive additive-type noise and the coefficients exhibit mean-field dependency. \begin{equation*} 
\begin{cases}
&dX(t)=\pr{A\pr{X(t)}+f\pr{t,X(t),\mathbb{P}_{X(t)}}}\ dt+g(t) dW(t),\ t\in \pr{0,T},\\
&X(0)=X_0,
\end{cases}
\end{equation*}where $A$ stands for the complex porous media operator, $f$ is the drift depending on the law of the solution $\mathbb{P}_{X(t)}$ and $g$ is the (additive) noise coefficient. We discuss the regularity assumptions on $f$ in a functional setting and provide examples of cylindrical constructions for such coefficients (see Section \ref{Sec4.2}).
Using these tools, combined with a Galerkin-type approach, we prove the consistency of our equations for strict monotonic $\beta$. This constitutes the aim of Section \ref{Sec4.3}. 
The presence of a further Lipschitz non-linearity $f$ is intended for future design of stabilizing controls, either in the area of controllability in the spirit of \cite{GM_2021}, or for state-constrained design, see \cite{CGM_2023_JEEQ}. The classical results are not directly applicable and we offer a detailed treatment of the existence.\\

\noindent In the last section (Section 5) we offer a different interpretation of the solution for the Schr\"{o}%
dinger-type porous-media equation through an optimal control problem inspired by the
variational principle in \cite{BrezisEkeland_1976} or again \cite{N_1976}. While in the real case the arguments are connected to Fenchel duality and sub-differential expressions of monotone operators, the complex setting requires to employ variations of Fitzpatrick's characterization of maximal montonicity, cf. \cite{Fitz_1988}. We offer a detailed treatment of the Brézis-Ekeland-like variational characterization for strictly monotonic $\beta$ leading to Proposition \ref{PropBE}. A different characterization based on the induced porous media operator $-\Delta\beta$ being also maximal monotonic is hinted at in Remark \ref{RemBE}. \\

\noindent To summarize, the main contributions of the paper rely in 
\begin{itemize}
\item the derivation of a new equation of Schrödinger type driven by porous media complex operators through Feynman-type arguments closely related to the quantum mechanics. We believe this gives a physical and philosophical reason for the study of such equations;
\item the complete study of such equations with Brownian additive perturbation motivated by measurement of errors and with law-depending non-linearities to emphasize possible entanglement of quantum wave instances;
\item the variational (or control) interpretation of the solution using Fitzpatrick-type representation of monotone operators. We believe that the independent study of such problems should allow a stand-alone method to prove the existence of solutions;
\end{itemize}

\section{Heuristics on the Model Following Feynman's Approach}\label{Sec2}

\noindent \textbf{The Lagrangian}

To understand the developments hereafter, we begin with recalling some elements constituting Feynman's path approach in \cite{F_1948} as an alternative to Heisenberg-Dirac, see \cite{Dirac_1930} or Schrödinger's approach \cite{Schrodinger_1926}. The approach can be split into two \emph{postulates}.
\begin{enumerate}
\item The most important is the fact that \emph{"the paths contribute equally in magnitude, 
but the phase of their contribution is the classical 
action (in units of $\hbar$)"}. This leads, for Feynman, to contributions of a path $\pp{x(t)}_{t\in \mathbb{R}}$ of type
$\exp\pr{\frac{i}{\hbar}S\pp{x(t)}}$. At the same time, the action is to be linked to a Lagrangian depending on the speed
and on the position of the point $L$ and, as such, the action is determined by $S\pp{x(t)}:=\int L(\overset{\cdot}{x}(t),x(t))dt$.
\item The remaining postulate requires, as always in quantum analysis, the "superposition of probability amplitudes". Roughly speaking, given measurements $a$, resp. $c$ of events $A$ and $C$, one sums, over intermediate mutually excluding events to get  
$\phi_{ac}=\sum_{b}\phi_{ab}\phi_{bc}$. When combined with discretized trajectory $\pp{x_j:=x(t_j)}_{j\in\mathbb{Z}}$, this leads (see \cite[Eq. (9)]{F_1948}) to 
$\phi(R)\approx\int_R \exp\pp{\frac{i}{\hbar}\sum_{-\infty<j<\infty}S\pr{x_{j+1},x_j}}\ldots \frac{dx_{j+1}}{A}\frac{dx_{j}}{A}\ldots$, given a region $R$. The parameter $A$ is a normalization one.
\end{enumerate}
As a consequence of these postulates, the \emph{wave function} computed on the $t_k$-non anticipating sub-region $R'$ is given by \[\psi(x_k,t)\approx \int_{R'} \exp\pp{\frac{i}{\hbar}\sum_{-\infty<j<k}S\pr{x_{j+1},x_j}}\ldots \frac{dx_{k-2}}{A}\frac{dx_{k-1}}{A}\frac{1}{A}.\]
Then Schrödinger's original  equation is obtained with the usual Lagrangian in a movement-against-$V$-potential i.e.  $L(\overset{\cdot}{x},x):=\frac{m\pr{\overset{\cdot}{x}}^2}{2}-V(x)$.

\bigskip

Let us give \textbf{another way of interpreting these arguments}.  Let us first fix a Lipschitz-continuous (for now) real function $\beta :\mathbb{%
R\rightarrow R}$ and denote by $\overset{\cdot }{\beta }$ its $L^{\infty
}\left( \mathbb{R}\right) $ almost everywhere derivative.  In our case, we shall consider the $L$ Lagrangian of the
following form 
\begin{equation*}
L^{0}\left( p,x\right) :=\overset{\cdot }{\beta }(x) p.
\end{equation*}

\noindent Following the idea from Feynman \cite{F_1948}, one can describe the path of
a free particle by a straight-line and therefore the energy on $\pp{t_j,t_{j+1}=t_j+\varepsilon}$ is approximated (provided $\overset{\cdot }{\beta }$ is continuous) as 
\begin{eqnarray}
S^{\varepsilon }\left( x_{j+1},x_{j}\right)  &=&\frac{\varepsilon }{2}\left(
L\left( \frac{x_{j+1}-x_{j}}{\varepsilon },x_{j+1}\right) +L\left( \frac{%
x_{j+1}-x_{j}}{\varepsilon },x_{j}\right) \right) \medskip   \label{S0_1} \\
&=&\frac{\overset{\cdot }{\beta } \left( x_{j+1}\right) +\overset{\cdot }{\beta } \left( x_{j}\right) }{2}\left(
x_{j+1}-x_{j}\right) ,  \notag
\end{eqnarray}%
or, in a simpler formulation, as%
\begin{equation}
S^{\varepsilon }\left( x_{j+1},x_{j}\right) =\varepsilon L\left( \frac{%
x_{j+1}-x_{j}}{\varepsilon },\frac{x_{j+1}+x_{j}}{2}\right) =\overset{\cdot }{\beta } \left( 
\frac{x_{j+1}+x_{j}}{2}\right) \left( x_{j+1}-x_{j}\right) .  \label{S0_2}
\end{equation}

\noindent Feynman mentions such forms of the Lagrangian with emphasis on the
difference in scales, see \cite[Page 376]{F_1948}. Particular emphasis
is put on \eqref{S0_2} for the symmetry in the expected Hamiltonian. 

In Feynman's
interpretation, given a path $\left( t_{k},x_{k}\right) _{k\in 
%TCIMACRO{\U{2124} }%
%BeginExpansion
\mathbb{Z}
%EndExpansion
}$ the transition 
\begin{equation*}
\left( ...,t_{j-1},x_{j-1},t_{j},x_{j}\right) \longrightarrow \left(
...,t_{j-1},x_{j-1},t_{j},x_{j},t_{j+1},x_{j+1}\right) 
\end{equation*}
is Markovian, homogeneous (independent of the number of transitions)
stationary (only depending on the time available $t_{k}-t_{k-1}$ but not on $%
t_{k-1}$) and governed by a non-random $S$ as before, hence leading to%
\begin{eqnarray*}
\psi \left( x_{k},t_{k}\right)  &:&=\int e^{\frac{i}{\hbar}\sum_{j=-\infty
}^{k-1}S_{j+1}\left( t_{j+1},x_{j+1},t_{j},x_{j},...\right) }\frac{dx_{k-1}}{%
A}\frac{dx_{k-2}}{A}...\medskip  \\
&=&\mathbb{E}\left[ e^{\frac{i}{\hbar}\sum_{j=-\infty }^{k-1}S_{j+1}\left(
t_{j+1},X_{j+1},t_{j},X_{j},...\right) }\right] \medskip  \\
&=&\mathbb{E}\left[ \mathbb{E}\left[ e^{\frac{i}{\hbar}\sum_{j=-\infty
}^{k-1}S_{j+1}\left( t_{j+1},X_{j+1},t_{j},X_{j},...\right) }\left\vert 
\mathcal{F}_{k-1}\right. \right] \right] \medskip  \\
&=&\mathbb{E}\left[ \mathbb{E}\left[ e^{\frac{i}{\hbar}{\color{black} S_{k}}\left(
t_{k},x_{k},t_{k-1},X_{k-1},...\right) }\left\vert \mathcal{F}_{k-1}\right. %
\right] \psi \left( X_{k-1},t_{k-1}\right) \right] \medskip  \\
&=&\mathbb{E}\left[ e^{\frac{i}{\hbar}S^{t_{k}-t_{k-1}}\left(
x_{k},X_{k-1}\right) }\psi \left( X_{k-1},t_{k-1}\right) \right] ,
\end{eqnarray*}%
where $A$ is a factor whose value we shall determine later and $\hbar$ is the reduced Planck constant. 

Although obvious, it is maybe worth mentioning that $X_{j}$ have $\frac{%
dx_{j}}{A}$ densities if {\color{black} $j\leq k-1$} while $X_{k}=x_{k}$ is fixed and $%
\mathcal{F}$\ is the naturally induced filtration associated to these random
variables. {\color{black} Furthermore, $S$ can depend on the number of transition, hence the sub-index $j+1$. } We also emphasize that the last equality only holds for the stationary, Markovian case which is considered in the original Feynman computations.

\bigskip 

\noindent \textbf{Complex developments}

\bigskip 

Since the previous developments are done in $\mathbb{R}$, we are trying to
slightly generalize in two directions: by changing the space to $\mathbb{C}$
and by considering another form of path dependence.

The reader is invited to recall that the complex space $\C^n$ has a Hilbert structure with the usual scalar product $\scal{z_1,z_2}=\sum_{1\leq j\leq n}\overline{z_{1,j}}z_{2,j}$ which is (left-)sesquilinear. As usual, $z=\pr{z_j}_{1\leq j\leq n}$ is a column $n$-dimensional vector over $\mathbb{C}$. 

We consider $\beta :=\beta _{1}+i\beta _{2},$ $\beta _{k}:%
\mathbb{R}^{2}\rightarrow \mathbb{R},  \ k=1,2, $ which is (complex-)differentiable. \\
 
\noindent We now turn to a model that assumes perhaps the simplest form of \emph{path dependence} in the formulation \eqref{S0_2}, i.e.
\begin{equation}\label{EqSComplex}
S(z,s,z')=\beta'\pr{\frac{\psi(z',s)+\psi(z,s)}{2}}\pr{\psi(z,s)-\psi(z',s)}.
\end{equation}
Of course, in this case, the derivative is computed in the complex sense and we indicated this by writing $\beta'$. Heuristically speaking,  the reader is invited to notice the following.

\begin{enumerate}
\item The distance is not taken in Euclidean metric between positions $x$ and $y$ but along the wave $\psi(\cdot,s)$ with the time $s$ fixed;
\item Stationarity is no longer enforced as the "starting" time $s$ enters the expression;
\item Since $\psi$ represents a memory of the path (as it is integrated along past events), the computations are, in some sense, path-dependent; This remark also justifies considering, in the equation, nonlinear terms with law-dependency.
\item One is actually looking into a fixed point of \begin{equation}\label{System}
\begin{cases}
S(x,t,y):=\beta'\pr{\frac{\psi(y,t)+\psi(x,t)}{2}}(\psi(x,t)-\psi(y,t));\\
\psi(x_0=x,\ t_0=t):=\lim_{\varepsilon\rightarrow 0}\int e^{\frac{i}{\hbar}\sum_{n\geq 0}S(x_n,t_n,x_{n+1})}\frac{dx_{n+1}}{A}\frac{dx_{n+2}}{A}\cdots;\\
0<\varepsilon:=\sup_{n\in\mathbb{Z}}\pr{t_{n+1}-t_{n}};
\end{cases}
\end{equation}
\item The integrals can be taken along the same points as in the aforementioned formula for $\psi(x_k,t_k)$ but involving copies $\tilde{X}$ independent of $X$. This strengthens the aforementioned intuition that the entanglement may also translate into the presence of mean-field or McKean-Vlasov terms in the resulting equation.
\end{enumerate}
We emphasize that, at this point, we are not going to prove the consistency of \eqref{System} but merely use this in order to deduce the associated \emph{Schrödinger-type equation}.\\
We have \begin{equation}
\psi(x,t+\varepsilon)=\int e^{-\frac{i}{\hbar}\beta'\pr{\frac{\psi(x+\xi,t)+\psi(x,t)}{2}}\pp{\psi(x+\xi,t)-\psi(x,t)}}\psi(x+\xi,t)\frac{d\xi}{A}.
\end{equation}
It is expected that the normalization $A$ is of order of $\varepsilon^{1/2}$ and so is the support of $\xi$ .  For simplicity, this support is taken $\pp{-\sqrt{\varepsilon},\sqrt{\varepsilon}}$ with the obvious (uniform) normalization $A:=2\sqrt{\varepsilon}$. Since we are going to rely on approximations up to order $\varepsilon$, we ignore the terms over $\xi^2$.  Baring this in mind, we have
\begin{equation}\begin{split}
&\psi(x,t+\varepsilon)\\
\simeq\int_{\pp{-\sqrt{\varepsilon},\sqrt{\varepsilon}}} &\exp\pr{\frac{i}{\hbar}\bigg\{-\frac{\beta''(\psi(x,t))}{2}\pp{\partial_x\psi(x,t)\xi}^2-\beta'(\psi(x,t))\pp{\partial_x\psi(x,t)\xi+\frac{\xi^2}{2}{\color{black} \partial_{x}^2}\psi(x,t)}\bigg\}}\\&\times\pp{\psi(x,t)+\partial_x\psi(x,t)\xi+\frac{\xi^2}{2}{\color{black} \partial_{x}^2}\psi(x,t)}\frac{d\xi}{2\sqrt{\varepsilon}}\\
\simeq\int_{\pp{-\sqrt{\varepsilon},\sqrt{\varepsilon}}} &\pr{1-\frac{i}{2\hbar}\Delta\beta(\psi(x,t))\xi^2-\frac{i}{\hbar}\nabla\beta(\psi(x,t))\xi-\frac{1}{2\hbar^2}\pr{\nabla\beta(\psi(x,t))}^2\xi^2}
\\&\times\pp{\psi(x,t)+\partial_x\psi(x,t)\xi+\frac{\xi^2}{2}{\color{black} \partial_{x}^2}\psi(x,t)}\frac{d\xi}{2\sqrt{\varepsilon}}.
\end{split}
\end{equation}
We get
\begin{equation}
\begin{split}
&\varepsilon\partial_t\psi(x,t)\simeq \\&\frac{\varepsilon}{6}\pp{\psi(x,t)\pr{-\frac{i}{\hbar}\Delta\beta(\psi(x,t))-\frac{1}{\hbar^2}\pr{\nabla\beta(\psi(x,t))}^2}+\Delta \psi(x,t)-2\frac{i}{\hbar}\nabla\beta(\psi(x,t))\nabla \psi(x,t)}.
\end{split}
\end{equation}
We now set $\beta(e^\cdot)=\beta^0\pr{\cdot}$ and set $\phi:=\log\psi$ to get 
\begin{equation}\label{EcSchro}
\begin{split}
&\partial_t\phi(x,t)\simeq \\&
\frac{1}{6}\pp{-\frac{i}{\hbar}\Delta\beta^0\pr{\phi(x,t)}-\frac{1}{\hbar^2}\pr{\nabla \beta^0\pr{\phi(x,t)}}^2+\Delta\phi(x,t)+\pr{\nabla\phi(x,t)}^2-2\frac{i}{\hbar}\nabla\beta^0\pr{\phi(x,t)}\nabla\phi(x,t)}\\
=&\frac{1}{6}\pp{\Delta\pr{-\frac{i}{\hbar}\beta^0+Id}\pr{\phi(x,t)}+\pr{\nabla\pr{-\frac{i}{\hbar}\beta^0+Id}\pr{\phi(x,t)}}^2}\\
=&\pp{\Delta\tilde{\beta}\pr{\phi(x,t)}+\gamma\pr{\nabla\tilde{\beta}\pr{\phi(x,t)}}^2},
\end{split}
\end{equation}
where $\tilde{\beta}=\frac{1}{6}\pr{-\frac{i}{\hbar}\beta^0+Id}$ and $\gamma=6> 0$.\\

The fundamental example one has in mind is $\tilde{\beta}(z)=\kappa iz$ and a classical change allows one to obtain Schrödinger's equation from \eqref{EcSchro}. Indeed, in this case, we are dealing with
\[\partial_t\phi(x,t)=\kappa i \Delta \phi(x,t)-\gamma\kappa^2 \pr{\nabla \phi(x,t)}^2,\] and by setting $\Phi:=\frac{e^{i\kappa \gamma \phi}}{\kappa \gamma}$, one computes 
\begin{equation*}
\partial_t\Phi=i\kappa \gamma\Phi \partial_t\phi;\ \nabla \Phi=i\kappa \gamma \Phi \nabla\phi ;\ \Delta \Phi=i\kappa\gamma \Phi\Delta\phi -\kappa^2\gamma^2  \Phi\pr{\nabla\phi}^2,
\end{equation*}which amounts to 
\begin{equation}
\label{EcSchro0}
\partial_t\Phi(x,t)=i\kappa\Delta\Phi=\Delta\tilde{\beta}(\Phi(x,t)),
\end{equation}
which is the classical Schrödinger's equation.\\
\noindent This trick and observation is also used for \cite{Porretta1999},  (see also \cite[Section 4]{PV_2012}) to treat an equation of form \eqref{EcSchro} with the usual Laplace operator instead of the porous media one from 
\begin{equation}\label{EcSchro22}
\begin{cases}
\partial_t\phi(x,t)=\Delta \beta\pr{\phi(x,t)}+f(t),&\textnormal{ on }\mathcal{O}\times\pr{0,T},\\
\phi(0,\cdot)=\phi_0,&\textnormal{ on }\mathcal{O}\times\set{0},\\
\phi=0,&\textnormal{ on }\Gamma:=\partial\mathcal{O}.
\end{cases}
\end{equation}

\section{Monotonicity Approach in $\mathbb{C}$ to the Existence and Uniqueness of Solution}\label{Sec3}

In this section we shall present an adaptation to $\mathbb{C}$ of different notions and results which are necessary to study equation \eqref{EcSchro22} using a monotonicity approach.

\begin{definition}\label{DefMon}
\begin{enumerate}
\item A function $\beta:\mathbb{C}\longrightarrow\mathbb{C}$ is said to be \emph{monotone} if there exists $\alpha\in\mathbb{R}_+$ such that  \[\Re\scal{\beta(z)-\beta(z'),z-z'}_\mathbb{C}\geq \alpha\abs{z-z'}^2,\ \forall z,z'\in\mathbb{C}.\]
When $\alpha>0$, the function is called \emph{strictly monotone}.
\item Given a complex linear vector space $\pr{V,\norm{\cdot}_V}$ whose dual is denoted by $V^*$, the functional $B:D(B)\subset V\longrightarrow V^*$ is said to be \emph{monotone} if there exists $\alpha\in\mathbb{R}_+$ such that  \[\Re\scal{B(h)-B(h'),h-h'}_{V^*,V}\geq \alpha \norm{h-h'}_V^2,\ \forall h,h'\in D(B),\]with $D$ classically denoting the domain of $B$.
\end{enumerate}
\end{definition}
\begin{remark}
Note that if $\beta$ is monotone, then \begin{align*}
\abs{z-z'+a(\beta(z)-\beta(z')}^2& \geq
\abs{z-z'}^2+2a\Re\scal{\beta(z)-\beta(z'),z-z'}\\
&\geq \abs{z-z'}^2+2a\alpha{\color{black}\abs{z-z'}}^2\\
&\geq\abs{z-z'}^2,
\end{align*}for all $a\in\mathbb{R}_+$.  Due to the uniform convexity of $\mathbb{C}$ the two notions coincide in this framework (see \cite[(M) and (M')]{Kato_1967}).
\end{remark}

One can easily see in the following examples a way of understanding the notions above. 

\begin{example}
\begin{enumerate}
\item For every $p\geq 1$, the function $\beta(z):=i\abs{z}^{p-1}z$ is monotone in the sense of the previous definition. Indeed, \begin{align*}
&\Re\pr{\overline{\beta(z_1)-\beta(z_2)}(z_1-z_2)}\\&=\Re\pp{\pr{-i\pr{\abs{z_1}^{p-1}\Re z_1-\abs{z_2}^{p-1}\Re z_2}-\pr{\abs{z_1}^{p-1}\Im z_1-\abs{z_2}^{p-1}\Im z_2}}\pr{z_1-z_2}}\\
&=\pr{\abs{z_1}^{p-1}\Re z_1-\abs{z_2}^{p-1}\Re z_2}\pr{\Im z_1-\Im z_2}-\pr{\abs{z_1}^{p-1}\Im z_1-\abs{z_2}^{p-1}\Im z_2}\pr{\Re z_1-\Re z_2}\\
&=0.
\end{align*}
\item With the same argument, one can actually show that $\beta(z):=\tilde{\beta}(\abs{z})z$ is monotone, for $\tilde{\beta}:\mathbb{R}_+\longrightarrow\mathbb{C}$ such that $\Re\tilde{\beta}$ has non-negative values.  If $\Re\tilde{\beta}$ is bounded from below away from $0$, then the monotonicity is strict (i.e. $\alpha=\inf_{x\in\mathbb{R}_+}\Re\tilde{\beta}(x)>0$).
\end{enumerate}
\end{example}

In the present work we are essentially interested in a complex form of the porous media operator.  For this reason,  in connection to the afore-mentioned functions $\beta$,  we will introduce the \emph{porous media operator} $\Delta\beta$. 

Given an open bounded set $\mathcal{O}\subset \mathbb{R}^n$,  we introduce the following Sobolev spaces which are necessary for the construction of the complex porous media operator:
\begin{itemize}

\item $\mathbb{L}^2=\set{\phi:\mathcal{O}\longrightarrow \mathbb{C}:\ \Re\phi,\Im \phi\in \mathbb{L}^2\pr{\mathcal{O};\mathbb{R}}}$ with the Hermitian inner product
\[\scal{\phi,\psi}_{\mathbb{L}^2}:=\int_\mathcal{O}\overline{\phi(\xi)}\psi(\xi)d\xi,\ \phi,\psi\in\mathbb{L}^2,\]and the induced Euclidean norm.
\item $H_0^1:=\set{\phi:\mathcal{O}\longrightarrow \mathbb{C}:\ \Re\phi,\Im \phi\in H_0^{1,2}\pr{\mathcal{O};\mathbb{R}}}$, for the classical \emph{Sobolev spaces} of order $1$ in $\mathbb
L^2\pr{\mathcal{O};\mathbb{R}}$ where Dirichlet boundary conditions are enforced. The scalar product is the usual one involving $\mathbb{C}$-linear spaces i.e.
\[\scal{\phi,\psi}_{H_0^1}:=\int_\mathcal{O}\scal{\nabla\phi(\xi),\nabla\psi(\xi)}_{\mathbb{C}^n}d\xi,\ \phi,\psi\in H_0^1.\] Furthermore, the induced norm will be denoted by $\norm{\cdot}_{H_0^1}$. 

\end{itemize}

The reader is invited to note that this is possible because $\mathcal{O}$ is assumed to be bounded and due to Poincaré's inequality (otherwise,  $\scal{\phi,\psi}_{\mathbb{L}^2}$ should be added).
One considers $\pr{H_0^1}^*$ to be the dual of $H_0^1$.  The duality (such that $\pr{H_0^1,\mathbb{L}^2,\pr{H_0^1}^*}$ is a \emph{Ghelfand triple} when $\mathbb{L}^2$'s dual is identified with itself via Riesz's principle) yields
\[\scal{\phi,\psi}_{\pr{H_0^1}^*,H_0^1}=\int_{\mathcal{O}}\overline{	\phi(\xi)}\psi(\xi)d\xi.\]

As a consequence,  and by using integration-by-parts arguments, whenever $\phi,\psi\in H_0^1$, one has\[\scal{-\Delta\phi,\psi}_{\pr{H_0^1}^*,H_0^1}=\int_{\mathcal{O}}\scal{\nabla\phi(\xi),\nabla\psi(\xi)}_{\mathbb{C}^n}d\xi=\scal{\phi,\psi}_{H^1_0}.\] As a consequence, $-\Delta:H_0^1\longrightarrow \pr{H_0^1}^*$ provides an isometric isomorphism.  The Hilbert structure on $\pr{H_0^1}^*$ is given by \[\scal{\Delta\phi,\Delta\psi}_{\pr{H_0^1}^*}=\scal{\phi,\psi}_{H_0^1}, \ \forall \phi,\psi\in H_0^1.\]
Now one can identify $\pr{H_0^1}^*$ and $H_0^1$ via $(-\Delta)^{-1}$ and use this identification in order to obtain the Gelfand triple\begin{equation}
\label{Gelfand}
V:=\mathbb{L}^2\subset \pr{H_0^1}^*\subset V^*=(\mathbb{L}^2)^*
\end{equation}
where $(\mathbb{L}^2)^*$ is the dual of $\mathbb{L}^2$ with the $ \pr{H_0^1}^*$ topology.  We have chosen to keep the usual complex inner product in the definitions (in the physical sense) instead of the real one in which a symmetric $\Re\scal{\cdot,\cdot}$ would be considered. This does not impact the induced norm, nor the differential formulae which are still written with $\Re$ appearing.\\

Throughout the paper,  we will keep this notation in order not to perturb our readers when reading $\pr{\mathbb{L}^2}^*$ as defined above.
It is then immediate that $-\Delta$ extends to an isometry from $V$ to $V^*$ such that \begin{equation}
\scal{-\Delta\phi,\psi}_{V^*,V}=\scal{\phi,\psi} _{\mathbb{L}^2}, 
\end{equation}for all $\phi,\psi\in V=\mathbb{L}^2$.
All these arguments follow in the same way as their real-spaces analogous. The interested reader is invited to take a look at \cite[Example 4.1.7, Example 4.1.11, Lemmas 4.1.12; 4.1.13]{PR_2007} for the real-valued settings.

\begin{definition}
For an operator $\beta:\mathbb{C}\longrightarrow\mathbb{C}$ which is monotone and has linear growth i.e. 

\begin{equation}\label{Asslingrowth}\exists c\in\mathbb{R}_+\textnormal{ s.t. }\abs{\beta(z)}\leq c\pr{1+\abs{z}},\ \forall z\in\mathbb{C}.
\end{equation}
we can properly define the {\emph{ complex porous media operator}} by
 \begin{align*}
A:=\Delta\beta:V\longrightarrow V^*, 
A(\phi):=\Delta\beta(\phi),\ \forall \phi\in \mathbb{L}^2.
\end{align*}
\end{definition}

One can check that $-A:=-\Delta\beta$ is monotone on $\pr{H_0^1}^*$.
Indeed,\begin{align*}\Re\pr{ \scal{-A(\phi)+A(\psi),\phi-\psi}_{V^*,V}}&=\Re\pr{ \int_{\mathcal{O}}\scal{\beta(\phi(\xi))-\beta(\psi(\xi)),\phi(\xi)-\psi(\xi)}_{\mathbb{C}}d\xi}\\
&=\int_{\mathcal{O}}\Re \scal{\beta(\phi(\xi))-\beta(\psi(\xi)),\phi(\xi)-\psi(\xi)}_{\mathbb{C}}d\xi\\&\geq \alpha\int_{\mathcal{O}}\abs{\phi(\xi)-\psi(\xi)}^2d\xi=\alpha\norm{\phi-\psi}_{\mathbb{L}^2}^2.\end{align*}
Note that strict monotonicity of $\beta$ implies strict monotonicity of $-A$.

\begin{example}If one assumes the particular form $\beta(z)=\tilde{\beta}(\abs{z})z$ for $\tilde{\beta}$ defined as in the examples above,  then $\beta(\bar{z})=\tilde{\beta}(\abs{z})\bar{z}$ and therefore  \begin{equation}\label{Abar}\overline{A(\phi)}=\overline{\Delta\beta(\phi)}=\Delta\Re \beta(\phi)-i\Delta\Im\beta(\phi)=\Delta\tilde\beta(\abs{\phi})\pr{\Re\phi-i\Im\phi}=A\pr{\bar{\phi}},\end{equation}for all $\phi\in\mathbb{L}^2$.
\end{example}

\begin{remark}\label{RemFullRangeandRiesz}
\begin{enumerate}
\item A complex function $\beta$ can be identified with a function $\begin{pmatrix}
\Re \beta\\\Im\beta
\end{pmatrix}:\mathbb{R}^2\longrightarrow\mathbb{R}^2$ and the monotonicity requirement is just

\begin{equation}
\scal{\begin{pmatrix}
\Re \beta(x_1,x_2)  \\\Im\beta(x_1,x_2)
\end{pmatrix}  - \begin{pmatrix}
\Re \beta(y_1,y_2) \\\Im\beta(y_1,y_2)
\end{pmatrix}   ,\begin{pmatrix}
x_1 - y_1\\x_2 - y_2
\end{pmatrix}}_{\mathbb{R}^2}\geq 0
\end{equation}

thus reducing to the monotonicity on $\mathbb{R}^2$. It is for this reason that the functional arguments (on monotonicity, for instance) can be conducted on real Hilbert spaces. \\
For instance,  if $\beta$ is continuous and strictly monotone, then it has full range $\mathbb{C}$. Indeed, using the above identification, $\begin{pmatrix}
\Re\beta\\\Im\beta
\end{pmatrix}$ is maximal monotone on the real Hilbert space $\mathbb{R}^2$ (one can apply, for instance,  \cite[Theorem 2.4]{Barbu_2010}). On the other hand, strict monotonicity implies coercivity in $\mathbb{R}^2$. As such (e.g.  \cite[Corollary 2.2]{Barbu_2010} for a version of the Minty-Browder Theorem), $\begin{pmatrix}
\Re\beta\\\Im\beta
\end{pmatrix}$ has full range $\mathbb{R}^2$, or, equivalently, $\beta$ has full range $\mathbb{C}$.
\item The isometry between $V$ and $V^*$ exhibited in the previous example is surjective, see \cite[Remark 4.1.14]{PR_2007}. In this particular case, the duality map between $V$ and $V^*$, denoted by $J$ in \cite{Barbu_2010}(and $F$ in the introduction of \cite{Kato_1967}) is actually given by the Riesz isomorphism involving $(-\Delta)$ and one is actually using the monotonicity condition \cite[(M')]{Kato_1967} in Definition \ref{DefMon}. 
\end{enumerate}
\end{remark}

We claim that the operator $-A:=-\Delta\beta$ is actually \emph{maximal monotone} (or, if one prefers, \emph{$m$-accretif}) when $\beta$ is strictly monotone, that is, we have the following result.
\begin{proposition}\label{PropMaxmonotone}
If $\beta$ is strictly monotone and has at most linear growth \eqref{Asslingrowth} as a $\mathbb{C}$-valued function, then the domain of $(J-A)^{-1}$ is the entire space $V^*$ (or, equivalently, $-A$ is maximal monotone on $(V,V^*)$). 
\end{proposition}
\begin{proof}
Although classical, this amounts to show that $\pr{J-\Delta\beta}(u)=v$ admits a solution $u\in \mathbb{L}^2(\mathcal{O})$ as soon as $v\in V^*$ (see \cite[Theorem 2.2]{Barbu_2010}). \\
On one hand, it easy to see that $\beta^{-1}$ is monotone and Lipschitz-continuous on $\mathbb{C}$ (please refer to Remark \ref{RemFullRangeandRiesz} for the domain).  Indeed, the strict monotonicity yields\[\Re\scal{x-y,\beta^{-1}(x)-\beta^{-1}(y)}\geq\alpha\abs{\beta^{-1}(x)-\beta^{-1}(y)}^2,\]with $\alpha>0$ yielding $\abs{\beta^{-1}(x)-\beta^{-1}(y)}\leq \frac{1}{\alpha}\abs{x-y}$. As a consequence, $\beta^{-1}$ provides a Lipschitz continuous,  bounded and monotone operator on $V$ and $J\beta^{-1}$ gives a continuous, bounded and monotone operator from $V$ to $V^*$. \\
The duality mapping given by $-\Delta$ is maximal monotone (see, for instance, \cite[Page 28]{Barbu_2010} for the real case, from which the complex one follows immediately).  It follows that the operator $J\beta^{-1}-\Delta$ is maximal monotone (as in \cite[Corollary 2.1]{Barbu_2010}). Furthermore, $J\beta^{-1}-\Delta$ is coercive since \[\frac{\scal{u,(J\beta^{-1}-\Delta)u}_{V,V^*}}{\norm{u}_V}\geq \pr{1+\frac{1}{\alpha}}\norm{u}_V.\]As such (see \cite[Corollary 2.2]{Barbu_2010}), the equation $\pr{J\beta^{-1}-\Delta}u=v$ admits a solution $u\in V$ whenever $v\in V^*$ and $\bar{u}:=\beta(u)$ satisfies $\pr{J-\Delta\beta}(u)=v$.
\end{proof}
\begin{remark}
\begin{enumerate}
\item The reader is invited to note that $z\mapsto \beta(z):=i\abs{z}^{p-1}z$ is also surjective. Indeed, if $w\in\mathbb{C}$, and $w\neq 0$, then $z:=\frac{-iw}{\abs{w}^{\frac{p-1}{p}}}$ provides the solution to $\beta(z)=w$ (the case $w=0$ leads to $z=0$). As such, one can define $\beta^{-1}(z):=\frac{-iz}{\abs{z}^{\frac{p-1}{p}}}$ and extend it to be $0$ at $z=0$ by continuity. The continuity of the $\mathbb{L}^q(\mathcal{O})$-induced operator requires a more careful consideration of the power $q$ and at least the Gelfand triple should be changed.  
\item It is clear that if $\beta$ is monotone, $z\mapsto \alpha\Re(z)+\beta(z)$ is strictly monotone as soon as $\alpha>0$.
\end{enumerate}
\end{remark}
\section{Existence and Uniqueness of the Solution }\label{Sec4}

In this section we shall prove the existence and uniqueness of the solution to an equation similar to the equation \eqref{EcSchro22} in a sense to be precisely defined hereafter. To account for errors of measurement, the equation will have an additive Brownian perturbation term. To account for independent copies as mentioned in the heuristics on Feynman's approach, we will also introduce a McKean-Vlasov (or mean-field) dependency in the coefficients. 
We are interested in the following equation

\begin{equation} \label{eccc}
\begin{cases}
&dX(t)=\pr{A\pr{X(t)}+f\pr{t,X(t),\mathbb{P}_{X(t)}}}\ dt+g(t) dW(t),\ t\in \pr{0,T},\\
&X(0)=X_0,
\end{cases}
\end{equation}
We need to introduce some notations.

\subsection{Notations}
\begin{enumerate}
\item Given a Hilbert space $H$, 
\begin{enumerate}
\item We denote by $\mathbb{P}_X$ the law of the $H$-valued random variable $X$ on a probability space $\pr{\Omega,\mathcal{F},\mathbb{P}}$;
\item We let $\mathcal{P}(H)$ stand for the family of probability measures on $H$, endowed with the natural Borelian $\sigma$-field;
\item The space $\mathcal{P}_2(H)$ will stand for the Wasserstein space of probability measures on $H$ with finite second-order moment. The 2-Wasserstein distance $W_{2,H}$ between $\mu,\nu\in\mathcal{P}_2\pr{H}$ is given, as usual, by
\[W_{2,H}^2(\mu,\nu)=\underset{\pi\in \mathcal{P}\pr{H\times H}, \pi(\cdot\times H)=\mu,\pi(H\times\cdot)=\nu}{\inf}\pr{\int_{H\times H}\norm{\xi-\eta}_{H}^2\pi\pr{d\xi,d\eta}}.\]
We will be interested in the special case of $H=\Hmo$ and we drop the dependency and only write $W_2$.
\item If $\mu,\nu$ are Borel-measurable functions on a time interval $\pp{0,T}$ taking their values in $\mathcal{P}_2(H)$, and $\gamma$ is a positive measure on $\pp{0,T}$ endowed with the Borel $\sigma$-field, we denote by $d_{\mathbb{L}^2\pr{\pp{0,T},\gamma;\mathcal{P}_2(H)}}$ the natural distance
\[d_{\mathbb{L}^2\pr{\pp{0,T},\gamma;\mathcal{P}_2(H)}}(\mu(\cdot),\nu(\cdot))=\int_0^TW_{2,H}\pr{\mu(t),\nu(t)}\ \gamma(dt).\]
\end{enumerate}
\item We let $\pr{\lambda_j,e_j}$ be the eigen-values/eigen-functions associated to $-\Delta$ on $\pr{H_0^1}\pr{\mathcal{O};\mathbb{R}}$ (with Dirichlet boundary conditions), the eigen-values being ordered increasingly.  These functions are taken in $\mathbb{L}^2\pr{\mathcal{O};\mathbb{R}}$ (actually, in the domain of $-\Delta$, i.e., $H^2\pr{\mathcal{O};\mathbb{R}}\cap H_0^1\pr{\mathcal{O};\mathbb{R}}$) and form an orthonormal basis in $\pr{H_0^1}^*\pr{\mathcal{O};\mathbb{R}}$. It is straightforward that $\tilde{e}_j:=\frac{1}{\sqrt{\lambda_j}}e_j$ form an orthonormal basis in $\mathbb{L}^2\pr{\mathcal{O};\mathbb{R}}$. 
\item The cylindrical Wiener process $W$ on $\pr{H_0^1}^*\pr{\mathcal{O};\mathbb{R}}$ is given by $W(t)=\sum_{k\geq 1}W_k(t)e_k$, where $W_k$ are mutually independent Brownian on the space $\pr{\Omega,\mathcal{F},\mathbb{F},\mathbb{P}}$
\end{enumerate}
\subsection{Coefficients: Assumptions and Examples}\label{Sec4.2}
We are introducing the following.
\begin{assumption}\label{Ass1} 
\begin{enumerate}\item The drift coefficient $f:\mathbb{R}\times \Hmo\times \mathcal{P}_2(\Hmo)\longrightarrow\Hmo$ satisfies the following.
\begin{enumerate}
\item For every $t\geq 0$ and every $\mu\in \mathcal{P}_2(\Hmo)$, the function $f(t,\cdot,\mu)$ leaves invariant $\Lo{2}$;
\item For every $H\in\set{\Hmo,\Lo{2}}$, \[\norm{f(t,0,\mu)}_{H}\leq \norm{f}_0\pr{f_0(t)+\pr{\int_{\Hmo}\norm{\xi}_{\Hmo}^2\mu(d\xi)}^{\frac{1}{2}}},\]for every $\mu\in\mathcal{P}_2\pr{\Hmo}$, where $\norm{f}_0$ is a positive constant and $f_0(\cdot)\in \mathbb{L}^2\pr{\pp{0,T},\mathcal{L}eb;\mathbb{R}}$;
\item There exists a real constant $\pp{f}_1$ such that for every $H\in\set{\Hmo,\Lo{2}}$, 
\[\norm{f(t,x,\mu)-f(t,y,\nu)}_{H}\leq \pp{f}_1\pr{\norm{x-y}_H+W_{2,\Hmo}\pr{\mu,\nu}},\] for every $x,y\in H$, every $t\geq 0$, and every $\mu,\nu\in\mathcal{P}_2\pr{\Hmo}$. We have specified $W_{2,\Hmo}$ in order to avoid all confusion. 
\end{enumerate}
\item We consider the predictable coefficients $\Re g,\Im g:\mathbb{R}\longrightarrow\mathcal{L}_2\pr{\mathbb{L}^2\pr{\mathcal{O};\mathbb{R}};\mathcal{D}\pr{\pr{-\Delta}^{\gamma}}}$, for some $\gamma>\frac{n}{2}$ (where $\mathcal{O}\subset \mathbb{R}^n$), where $\mathcal{L}_2$ denote Hilbert-Schmidt operators. We assume that
\[\mathbb{E}\pp{\int_0^T\pp{\norm{\Re g(t)}^2_{\mathcal{L}_2\pr{\mathbb{L}^2\pr{\mathcal{O};\mathbb{R}};\mathcal{D}\pr{\pr{-\Delta}^{\gamma}}}}+\norm{\Im g(t)}^2_{\mathcal{L}_2\pr{\mathbb{L}^2\pr{\mathcal{O};\mathbb{R}};\mathcal{D}\pr{\pr{-\Delta}^{\gamma}}}}}\ dt}<\infty.\]
\end{enumerate}
\end{assumption}
 For $n\geq 1$, we set \[\Pi_k\phi=\Pi_k\Re\phi+i\Pi_k\Im\phi,  \ \  \forall  \phi \in  \pr{H_0^1}^*. \]
 
The reader will note that if $x\in\mathbb{L}^2\pr{\mathcal{O};\mathbb{R}}$ is real-valued, then\begin{equation}\label{ProjL2Hmo}\begin{split}
\Pi_kx=\sum_{1\leq j\leq k}\scal{e_j,x}_{\Hmo}e_j&=\sum_{1\leq j\leq k}\frac{1}{\lambda_j}\scal{-\Delta e_j,x}_{\Hmo}e_j=\sum_{1\leq j\leq k}\frac{1}{\lambda_j}\scal{e_j,x}_{\Lo{2}}e_j\\&=\sum_{1\leq j\leq k}\scal{\tilde{e}_j,x}_{\Lo{2}}\tilde{e}_j.
\end{split}\end{equation}In other words, $\Pi_k$ gives the same projection when one looks at $\Lo{2}$ and $\Hmo$ (by definition, also for $x\in\Lo{2}$ complex).\\
It is by now standard that $-\Pi_nA$ provides a Lipschitz operator both in $\pr{H_0^1}^*\pr{\mathcal{O}}$ and in $\mathbb{L}^2\pr{\mathcal{O}}$.  Let us provide some estimates we are going to use hereafter.
\begin{proposition}
The following assertions hold true.
\begin{enumerate}
\item If $x\in \Lo{2}$, then $x\in \Hmo$ and $\norm{x}_{\Hmo}^2\leq \norm{x}_{\Lo{2}}^2$;
\item If $x\in \Hmo$, then $\Pi_kx\in\Lo{2}$ and $\norm{\Pi_kx}_{\Lo{2}}^2\leq \lambda_k\norm{\Pi_kx}_{\Hmo}^2$;
\item If $\beta$ is Lipschitz-continuous in $\mathbb{C}$, then $\Pi_k\beta$ is Lipschitz-continuous on $\Pi_k\Hmo$ and ${\color{black}\Pi_k}\Lo{2}$;
\item If $\beta$ is Lipschitz-continuous in $\mathbb{C}$, then $\Pi_k\Delta\beta$ is Lipschitz-continuous on $\Pi_k\Hmo$ and, equivalently, on $\Pi_k\Lo{2}$;
\end{enumerate}
\end{proposition}
\begin{proof}
One has \begin{align*}
\Re x+i\Im x&=\sum_{k\geq 1}\pr{\scal{e_k,\Re x}_{\pr{H_0^1}^*\pr{\mathcal{O};\mathbb{R}}}e_k+i\scal{e_k,\Im x}_{\pr{H_0^1}^*\pr{\mathcal{O};\mathbb{R}}}e_k}=\sum_{k\geq 1}\scal{e_k,x}_{\Hmo}e_k.
\end{align*}
As such, \begin{align*}\norm{x}_{\Hmo}^2=\norm{\Re x}_{\pr{H_0^1}^*\pr{\mathcal{O};\mathbb{R}}}^2+\norm{\Im x}_{\pr{H_0^1}^*\pr{\mathcal{O};\mathbb{R}}}^2=\sum_{k\geq 1} \abs{\scal{e_k,x}_{\Hmo}}^2.\end{align*}
The same arguments apply to $\Lo{2}$.
\begin{enumerate}
\item The first inequality is a simple consequence of \eqref{ProjL2Hmo}.
\item One has 
\begin{align*}
\norm{\Pi_kx}_{\Lo{2}}^2=\sum_{1\leq j\leq k}\lambda_j\abs{\scal{e_j,x}_{\Hmo}}^2\leq\lambda_k\sum_{1\leq j\leq k}\abs{\scal{e_j,x}_{\Hmo}}^2\leq \lambda_k\norm{\Pi_kx}_{\Hmo}^2.
\end{align*}
\item If $x,y\in\Lo{2}$, then $\Pi_k\beta(x)\in\Pi_k\Lo{2}=\Pi_k\Hmo$ and similarly for $y$. Furthermore,
\begin{align*}
\norm{\Pi_k\beta(x)-\Pi_k\beta(y)}_{\Lo{2}}\leq \norm{\beta(x)-\beta(y)}_{\Lo{2}}\leq \pp{\beta}_1\norm{x-y}_{\Lo{2}};\\
\norm{\Pi_k\beta(x)-\Pi_k\beta(y)}_{\Hmo}\leq  \pp{\beta}_1\norm{x-y}_{\Lo{2}}\leq \lambda_k\pp{\beta}_1\norm{x-y}_{\Hmo},
\end{align*}
the last inequality being valid for $x,y\in\Pi_k\Hmo$.
\item The presence of $\Delta$ only accounts for a diagonal operator (hence bounded on the finite dimensional space $\Pi_k\Lo{2}$).
\end{enumerate}
\end{proof}

\begin{example}\label{Exp1}
The typical example we have in mind for the drift comes from mean-field control. Let $\pr{f_k}_{k\geq 1}$ be a family of Lipschitz continuous functions on $\mathbb{C}^2$, taking their values in $\mathbb{C}$, and such that $f(z,0)=f(0,z)=0,\ \forall z\in\mathbb{C}$. Their Lipschitz constants are denoted by $\pp{f_k}_1$.  Furthermore, consider a double sequence $\pr{\theta_{j,k}\in\mathbb{C}}_{1\leq j,k}$ such that
$\sum_{k\geq 1}\pp{f_k}_1^2\pr{\sum_{1\leq j\leq k} \abs{\theta_{k,j}}}^2<\infty.$
We set
\[f(x,\mu)=\sum_{k\geq 1}\sum_{1\leq j\leq k}\theta_{k,j}\scal{e_k,\int_{\Hmo}\scal{e_j,f_k(\bar{x},\Pi_jy)}_{\Hmo}\mu(dy)\mid_{\bar{x}=\Pi_kx}}_{\Hmo}e_k.\]
We further introduce $F_k(\bar{x},\mu):=\sum_{1\leq j\leq k}\theta_{k,j}\int_{\Hmo}\scal{e_j,f_k(\bar{x},\Pi_jy)}_{\Hmo}\mu(dy)$. It is easy to see that if $\bar{x},\bar{y}\in\mathbb{C}$, one has
\[\abs{F_k(\bar{x},\mu)-F_k(\bar{y},\mu)}\leq \sum_{1\leq j\leq k}\abs{\theta_{k,j}}\frac{\pp{f_k}_1\sqrt{\mathcal{L}eb(\mathcal{O})}}{\sqrt{\lambda_j}}\abs{\bar{x}-\bar{y}}.\]Indeed,
\begin{align*}&\abs{\scal{e_j,f_k(\bar{x},\Pi_jy)}_{\Hmo}-\scal{e_j,f_k(\bar{y},\Pi_jy)}_{\Hmo}}\\&=\frac{1}{\sqrt{\lambda_j}}\abs{\int_{\mathcal{O}}\pp{f_k(\bar{x},\Pi_jy(\xi))-f_k(\bar{y},\Pi_jy(\xi))}\tilde{e}_j(\xi)\ d\xi}\\
&\leq \frac{\pp{f_k}_1\abs{\bar{x}-\bar{y}}}{\sqrt{\lambda_j}}\int_{\mathcal{O}}\abs{\tilde{e}_j(\xi)} d\xi\leq  \frac{\pp{f_k}_1\abs{\bar{x}-\bar{y}}}{\sqrt{\lambda_j}}\times\sqrt{\mathcal{L}eb(\mathcal{O})},
\end{align*}from which the conclusion follows.\\

On the other hand, if $\mu,\nu\in\mathcal{P}_2\pr{\Hmo}$ and $\pi\in\mathcal{P}_2\pr{\Hmo\times \Hmo}$ is a coupling, then
\begin{align*}
&\abs{\int_{\Hmo}\scal{e_j,f_k(\bar{x},\Pi_jy)}_{\Hmo}\mu(dy)-\int_{\Hmo}\scal{e_j,f_k(\bar{x},\Pi_jy)}_{\Hmo}\nu(dy)}\\
&=\abs{\int_{\Hmo}\scal{e_j,f_k(\bar{x},\Pi_jy)-f_k(\bar{x},\Pi_jy')}_{\Hmo}\pi(dy,dy')}\\
&\leq \int_{\Hmo}\abs{\scal{e_j,f_k(\bar{x},\Pi_jy)-f_k(\bar{x},\Pi_jy')}_{\Hmo}}\pi(dy,dy')\\
&\leq \frac{1}{\sqrt{\lambda_j}}\int_{\Hmo}\abs{\scal{\tilde{e}_j,f_k(\bar{x},\Pi_jy)-f_k(\bar{x},\Pi_jy')}_{\Lo{2}}}\pi(dy,dy')\\
&\leq \frac{\pp{f_k}_1}{\sqrt{\lambda_j}}\int_{\Hmo}\norm{\Pi_jy-\Pi_jy'}_{\Lo{2}}\pi(dy,dy')\\
&\leq \pp{f_k}_1\int_{\Hmo}\norm{\Pi_jy-\Pi_jy'}_{\Hmo}\pi(dy,dy')\\
&\leq \pp{f_k}_1\pr{\int_{\Hmo}\norm{y-y'}^2_{\Hmo}\pi(dy,dy')}^{\frac{1}{2}}.
\end{align*}
Taking the minimum over such couplings $\pi$ yields
\begin{align*}&\abs{\int_{\Hmo}\scal{e_j,f_k(\bar{x},\Pi_jy)}_{\Hmo}\mu(dy)-\int_{\Hmo}\scal{e_j,f_k(\bar{x},\Pi_jy)}_{\Hmo}\nu(dy)}\\&\leq \pp{f_k}_1W_2\pr{\mu,\nu}.\end{align*}
It follows that \[\abs{F_k(\bar{x},\mu)-F_k(\bar{x},\nu)}\leq \pp{f_k}_1\sum_{1\leq j\leq k} \abs{\theta_{k,j}}W_2\pr{\mu,\nu}.\]
Then, we have that $f(x,\mu)=\sum_{k\geq 1}\scal{e_k,F_k(\Pi_kx,\mu)}_{\Hmo}e_k$, and we get
\begin{align*}
&\norm{f(x,\mu)-f(y,\nu)}_{\Hmo}^2=\sum_{k\geq 1}\abs{\scal{e_k,F_k(\Pi_kx,\mu)-F_k(\Pi_ky,\nu)}_{\Hmo}}^2\\
&\leq \sum_{k\geq 1}\frac{1}{\lambda_k}\abs{\scal{\tilde{e}_k,F_k(\Pi_kx,\mu)-F_k(\Pi_ky,\nu)}_{\Lo{2}}}^2\\
&\leq \sum_{k\geq 1}\frac{1}{\lambda_k}\norm{F_k(\Pi_kx,\mu)-F_k(\Pi_ky,\nu)}_{\Lo{2}}^2\\
&\leq \sum_{k\geq 1}\frac{1}{\lambda_k}\int_{\mathcal{O}}\pp{\sum_{1\leq j\leq k}\abs{\theta_{k,j}}\pp{f_k}_1\frac{\sqrt{\mathcal{L}eb(\mathcal{O})}}{\sqrt{\lambda_j}}\abs{\Pi_k(x)(\xi)-\Pi_k(y)(\xi)}+\pp{f_k}_1\sum_{1\leq j\leq k} \abs{\theta_{k,j}}W_2\pr{\mu,\nu}}^2\ d\xi\\
&\leq 2\sum_{k\geq 1}\frac{1}{\lambda_k}\pr{\sum_{1\leq j\leq k}\abs{\theta_{k,j}}}^2\pp{f_k}_1^2\pr{\frac{\mathcal{L}eb(\mathcal{O})}{\lambda_1}\norm{\Pi_kx-\Pi_ky}_{\Lo{2}}^2+\mathcal{L}eb(\mathcal{O})W_2^2\pr{\mu,\nu}}\\
&\leq 2\sum_{k\geq 1}\frac{1}{\lambda_k}\pr{\sum_{1\leq j\leq k}\abs{\theta_{k,j}}}^2\pp{f_k}_1^2\pr{\frac{\mathcal{L}eb(\mathcal{O})}{\lambda_1}\lambda_k\norm{x-y}_{\Hmo}^2+\mathcal{L}eb(\mathcal{O})W_2^2\pr{\mu,\nu}}\\
&\leq \frac{2\mathcal{L}eb(\mathcal{O})}{\lambda_1}\sum_{k\geq 1}\pr{\sum_{1\leq j\leq k}\abs{\theta_{k,j}}}^2\pp{f_k}_1^2\pr{\norm{x-y}_{\Hmo}^2+W_2^2\pr{\mu,\nu}}.
\end{align*}
The estimates in $\Lo{2}$ are in a similar line and simpler. Time dependence is obtained by adding a control-related parameter.
\end{example}
A simple look at the example before shows that it does not cover the simplest linear case $f(x,\mu)=x+\int_{\Hmo}y\ \mu(dy)$. This is the reason we propose the following second example.
\begin{example}
Let $(\theta^1_{j,k})_{1\leq j\leq m},(\theta^2_{j,k})_{1\leq j\leq m}$ be double sequences of complex numbers. We assume that, for every $k\geq 1$ the number $\alpha^l_{k}=card\set{j\geq 1\ :\ \theta^l_{j,k}\neq 0}<\infty$, for $l\in\set{1,2}$. Furthermore, we assume that $\pr{\underset{k\geq 1}{\sum}\alpha_{k}\frac{\lambda_{\max\set{j,k}}}{\lambda_j}\abs{\theta^l_{j,k}}^2}_{j\geq 1}$ is bounded in $\mathbb{R}$, for each $l\in\set{1,2}$,  and set
\begin{align*}
f(x,\mu)=\sum_{j,k\geq 1}\theta^1_{j,k}\scal{e_j,x}_{\Hmo}e_k+\int_{\Hmo}\sum_{j,k\geq 1}\theta^2_{j,k}\scal{e_j,y}_{\Hmo}e_k\ \mu(dy).
\end{align*}
Then $
f(x,\mu)-f(x',\mu)=\sum_{j,k\geq 1}\theta^1_{j,k}\scal{e_j,x-x'}_{\Hmo}e_k$ such that 
\begin{align*}
&\norm{f(x,\mu)-f(x',\mu)}_{\Hmo}^2=\sum_{k\geq 1}\abs{\scal{\sum_{j\geq 1}\theta^1_{j,k}e_j,x-x'}_{\Hmo}}^2\\
\leq &\sum_{k\geq 1}\alpha_{k}\sum_{j\geq 1}\abs{\theta^1_{j,k}}^2\abs{\scal{e_j,x-x'}_{\Hmo}}^2
\leq \sup_{j\geq 1}\pr{\sum_{k\geq 1}\alpha_{k}\abs{\theta^1_{j,k}}^2}\norm{x-x'}_{\Hmo}^2,
\end{align*}and
\begin{align*}
&\norm{f(x,\mu)-f(x',\mu)}_{\Lo{2}}^2=\sum_{k\geq 1}\lambda_k\abs{\scal{\sum_{j\geq 1}\theta^1_{j,k}e_j,x-x'}_{\Hmo}}^2\\
\leq &\sum_{k\geq 1}\alpha_{k}\sum_{j\geq 1}\frac{\lambda_k}{\lambda_j}\abs{\theta^1_{j,k}}^2\abs{\scal{\tilde{e}_j,x-x'}_{\Lo{2}}}^2
\leq \sup_{j\geq 1}\pr{\sum_{k\geq 1}\alpha_{k}\frac{\lambda_k}{\lambda_j}\abs{\theta^1_{j,k}}^2}\norm{x-x'}_{\Lo{2}}^2.
\end{align*}
The Lipschitz condition involving $W_2$ is quite similar.\\
This example includes finite dimensional projections (as it was already the case in the previous example, but also, for $m\in\mathbb{N}^*$ fixed, $m+1$-upper diagonal operators $\theta_{j,k}=\theta_{j,k}\mathbf{1}_{j\leq k\leq j+m}$,  provided $\frac{\lambda_{j+m}}{\lambda_j}\underset{j\leq k\leq j+m}{\sum}\abs{\theta_{j,k}}^2$ be bounded in $j$. In particular, diagonal operators are included if the diagonal elements are bounded in $\mathbb{C}$.
\end{example}

\subsection{The Equation}\label{Sec4.3}
As we have mentioned before,  the heuristics on the Feynman's approach suggest that the mean-field components (through independent copies, for instance) may be worth considering as part of the wave equation. We have chosen to do so by asking the drift coefficient $f$ to depend on the law of the solution. Furthermore, because of the well-known entanglement effects, the measures on the system (and the resulting differential wave equation) may not be precise. We model this through the presence of a Brownian-driven additive perturbation via the noise coefficient $g$. 
\begin{definition}\label{DefSol}

Let $X_0\in {\color{black}\pr{H_0^1}^*\pr{\mathcal{O}}}$ be an initial condition.  We call $X$ a strong solution to equation 
\begin{equation} \label{ecc}
\begin{cases}
&dX(t)=\pr{{\color{black}\Delta\beta}\pr{X(t)}+f\pr{t,X(t),\mathbb{P}_{X(t)}}}\ dt+g(t) dW(t),\ t\in \pr{0,T},\\
&X(0)=X_0,
\end{cases}
\end{equation}
if the following conditions are satisfied
\begin{itemize}
\item $X\in\mathbb{L}^2\pr{\Omega;C\pr{\pp{0,T};\Hmo}}\cap\mathbb{L}^2\pr{\Omega\times \pp{0,T};\Lo{2}}$;
\item $\beta(X(\cdot))\in \mathbb{L}^2\pr{\Omega\times\pp{0,T};H_0^1\pr{\mathcal{O}}}$;\\
$\int_0^\cdot\Delta\beta(X(s))ds\in\mathbb{L}^2\pr{\Omega;\mathbb{L}^\infty\pr{\pp{0,T};\pr{H_0^1}^*\pr{\mathcal{O}}}}$;\\
$X(t)=X_0+\int_0^t\Delta\beta\pr{X(s)}ds + \int_0^t f\pr{s,X(s),\mathbb{P}_{X(s)}}\ ds +\int_0^tg(s)\ dW(s)$, $\mathbb{P}\otimes dt$-a.s.
\end{itemize}
\end{definition}

We can give now the main result of this section
\begin{theorem}\label{ThMain}

Under the Assumptions \ref{Ass1},  {\color{black} and by assuming $\beta$ to be strictly monotone}, for each initial condition $X_0\in{\color{black}\pr{H_0^1}^*\pr{\mathcal{O}}}$, we have a unique solution to equation \eqref{ecc} in the sense of Definition \ref{DefSol}.

\end{theorem}

\begin{proof}
In a first step, we consider a fixed family of laws $\mu(t)$ and recall the stochastic convolution $W_g(t):=\int_0^tg(s)\ dW(s)$, for $t\geq 0$; for the desired regularity properties, the reader is referred to \cite{DaPratoZabczyk1992}. We are then dealing with a random-deterministic equation written on $Y(\cdot):=X(\cdot)-W_g(\cdot)$.
\begin{equation}\label{SDE2}
\begin{cases}
&dY(t)=\pp{A\pr{Y(t)+W_g(t)}+f\pr{t,Y(t)+W_g(t),\mu(t)}}\ dt,\ t\in \pr{0,T},\\
&X(0)=X_0,
\end{cases}
\end{equation}

For $\omega\in \Omega$ fixed from now on,  and with the notation $\tilde{f}\pr{t,y}=f(t,y,\mu(t))$, for $t\geq 0$, the equation 
\begin{equation}
\begin{cases}
&dY_{n}(t)=\pp{\Pi_n A\pr{Y_{n}(t))+\Pi_nW_g(t)}+\Pi_n \tilde{f}\pr{t,Y_{n}(t)+\Pi_n W_g(t)}}dt,\ t\in \pr{0,T},\\
&X_{n}(0)=\Pi_nX_0,
\end{cases}
\end{equation} 
is well-posed (through standard Lipschitz arguments) in $$C\pr{\pp{0,T};\Pi_n\pr{H_0^1}^*\pr{\mathcal{O}}}\cap\  \mathbb{L}^2\pr{\pp{0,T};\mathbb{L}^2\pr{\mathcal{O}}}.$$ 
It is also straightforward that this solution lives in $\Pi_n\pr{\pr{H_0^1}^*\pr{\mathcal{O}}}$ i.e.
\[\Pi_n{Y_{n}}=Y_{n}.\] With the same trick as before, it is easy to see that $Y_n$ is $X_n(\cdot)-\Pi_n W_g(\cdot)$ satisfying a usual SDE (set on a finite-dimensional space) and, hence, this can be taken as an adapted process. \\
Furthermore,  with the use of a differential formula (for instance, by taking a look at \cite[Lemma 1.3]{Kato_1967}), one gets
\begin{align*}
&\norm{Y_{n}(t)}_{\Hmo}^2-2\Re\int_0^t\scal{\Pi_nA\pr{Y_{n}(s)+\Pi_nW_g(s)},Y_{n}(s)}_{\Hmo} ds\\=&\norm{\Pi_nX_0}_{\Hmo}^2+2\int_0^t\Re \scal{\Pi_n\tilde{f}\pr{s,Y_{n}(s)+\Pi_nW_g(s)},Y_{n}(s)}_{\Hmo}ds.
\end{align*}
We recall that $-A=-\Delta\beta$, leading to\footnote{{Without loss of generality, in order to avoid constant complications, we assume that $\mathcal{R}\scal{\beta(z),z}\geq \abs{z}^2$ (i.e. the strict monotonicity constant of $\beta$ is $1$).}} \begin{align*}\scal{-\Pi_nA\pr{Y_{n}+\Pi_n W_g},Y_{n}}_{\Hmo}&=\scal{\Pi_n\beta\pr{Y_{n}+\Pi_n W_g},Y_{n}}_{\Lo{2}}\\ 
&\geq \scal{\beta\pr{Y_{n}},Y_{n}}_{\Lo{2}}-{\color{black}\pp{\beta}_1}\norm{\Pi_n W_g}_{\Lo{2}}\norm{Y_n}_{\Lo{2}}\\
&\geq \frac{1}{2} \scal{\beta\pr{Y_{n}},Y_{n}}_{\Lo{2}} {\color{black}-} \frac{1}{2}\norm{Y_{n}}^2_{\Lo{2}}-\frac{\pp{\beta}_1^2}{2}\norm{\Pi_nW_g}_{\Lo{2}}^2,\end{align*} and, by the hypotheses above,  \[\abs{\Re\scal{\tilde{f}(t,x+\Pi_nW_g(t)),x}_{\Hmo}}\leq C\pr{\norm{f(t,0,\mu(t))}_{\Lo{2}}^2+\norm{W_g(t)}_{\Lo{2}}^2+\norm{x}_{\Hmo}^2}.\] He have used here the embedding $\Lo{2}\subset\Hmo$ to pass from $\Hmo$ norms to $\Lo{2}$ ones for the relevant terms.

As a consequence, and keeping in mind that $C>0$ is a generic constant that can change from one line to another, but still remains independent of $t$, $\mu$ and $n\geq 1$,
\begin{equation}\label{Estim1.0}\begin{split}
&\norm{Y_{n}(t)}_{\Hmo}^2+\Re\int_0^t\scal{\beta\pr{Y_{n}(s)},Y_{n}(s)}_{\Lo{2}}ds\\&\leq \norm{X_0}_{\Hmo}^2+C\int_0^t\pr{\norm{f(s,0,\mu(s))}_{\Lo{2}}^2+\norm{W_g(s)}_{\Lo{2}}^2+\norm{Y_{n}(s)}_{\Hmo}^2}ds.
\end{split}\end{equation}{\color{black}The reader is invited to recall that $\Re\scal{\beta\pr{Y_{n}(s)},Y_{n}(s)}_{\Lo{2}}\geq \norm{Y_{n}(s)}_{\Lo{2}}^2$ such that the previous inequality implicitly provides estimates for $\norm{Y_{n}(s)}_{\Lo{2}}^2$.}
With a simple application of Gronwall's inequality,  we get 
\begin{equation}\label{Estim1}\begin{cases}
&(i)\ \norm{Y_{n}(t)}_{\Hmo}^2\leq e^{Ct}\pr{C\int_0^t\pr{\norm{f(s,0,\mu(s))}_{\Lo{2}}^2+\norm{W_g(s)}_{\Lo{2}}^2}\ ds+ \norm{X_0}_{\Hmo}^2},\\
&\begin{split}(ii)\ &\Re\int_0^t\scal{\beta\pr{Y_{n}(s)},Y_{n}(s)}_{\Lo{2}}ds\\&\leq e^{Ct}\pr{C\int_0^t\pr{\norm{f(s,0,\mu(s))}_{\Lo{2}}^2+\norm{W_g(s)}_{\Lo{2}}^2}\ ds+ \norm{X_0}_{\Hmo}^2},\end{split}\\
&\begin{split}(iii)\ &\int_0^t\norm{Y_{n}(s)}_{\Lo{2}}^2ds\\& \leq e^{Ct}\pr{C\int_0^t\pr{\norm{f(s,0,\mu(s))}_{\Lo{2}}^2+\norm{W_g(s)}_{\Lo{2}}^2}\ ds+ \norm{X_0}_{\Hmo}^2},\end{split}
\end{cases}\end{equation}for all $t\geq 0$. \\
Owing to the Lipschitz property of $f$ in $\Hmo$, we also get
\begin{align*}
&\int_0^t\norm{f(s,Y_n(s)+W_g(s),\mu(s))}_{\Hmo}^2\ ds\\&\leq 2\int_0^t\norm{f(s,0,\mu(s))}_{\Hmo}^2\ ds+4\pp{f}_1^2\int_0^t\pr{\norm{Y_n(s)}_{\Hmo}^2+\norm{W_g(s)}_{\Hmo}^2}\ ds.
\end{align*}
As a consequence, and along some subsequence, still denoted by $n$ for simplicity,
\begin{equation}
\begin{cases}
Y_{n}\textnormal{ converges to }Y\textnormal{ weakly * in }\mathbb{L}^\infty\pr{\pp{0,T};\Hmo};\\
Y_{n}\textnormal{ converges to }Y\textnormal{ weakly in }\mathbb{L}^2\pr{\pp{0,T};\Lo{2}};\\
\tilde{f}\pr{\cdot, Y_n(\cdot)+W_g(\cdot)}\textnormal{ converges to some }F\textnormal{ weakly in }\mathbb{L}^2\pr{\pp{0,T};\Hmo};\\
\beta\pr{Y_{n}+W_g}\textnormal{ converges to some }Z\textnormal{ weakly in }\mathbb{L}^2\pr{\pp{0,T};\Lo{2}}.
\end{cases}
\end{equation}
{\color{black} The latter weak convergence follows from the upper bounds of $\norm{Y_n}_{\mathbb{L}^2(\mathcal{O})}$, together with the Lipschitz property of $\beta$.}
Similarly, whenever $n,m\geq 1$, one has (assuming, without loss of generality, $m\geq n$), 
\begin{equation}\label{Estim2.0}\begin{split}
&\norm{Y_{n}(t)-Y_{m}(t)}_{\Hmo}^2\\&+2\Re\int_0^t\scal{\Pi_m\beta\pr{Y_{n}(s)+\Pi_mW_g(s)}-\Pi_m\beta\pr{Y_{m}(s)+\Pi_mW_g(s)},Y_{n}(s)-Y_{m}(s)}_{\Lo{2}} ds\\=&2\Re\int_0^t\scal{\Pi_n\tilde{f}\pr{s,Y_{n}(s)+\Pi_nW_g(s)}-\Pi_m\tilde{f}\pr{s,Y_{m}(s)+\Pi_mW_g(s)},Y_{n}(s)-Y_{m}(s)}_{\Hmo}ds\\&+2\Re\int_0^t\scal{(\Pi_m\beta\pr{Y_{n}(s)+\Pi_mW_g(s)}-\Pi_n\beta\pr{Y_{n}(s)+\Pi_nW_g(s)},Y_{n}(s)-Y_{m}(s)}_{\Lo{2}} ds\\
&+\norm{\Pi_nX_0-\Pi_mX_0}_{\Hmo}^2.
\end{split}\end{equation}
In order to simplify the understanding, let us specify how the different terms on the right-hand side are dealt with.
\begin{enumerate}
\item For the term containing $\tilde{f}$,  noting that $(\Pi_n-\Pi_m)Y_n=0$, and $(\Pi_n-\Pi_m)Y_m=(\Pi_n-I)Y_m$, we get
\begin{align*}
&\Re\scal{\Pi_n\tilde{f}\pr{s,Y_{n}(s)+\Pi_nW_g(s)}-\Pi_m\tilde{f}\pr{s,Y_{m}(s)+\Pi_mW_g(s)},Y_{n}(s)-Y_{m}(s)}_{\Hmo}\\
&=\Re\scal{(\Pi_n-\Pi_m)\tilde{f}\pr{s,Y_{n}(s)+W_g(s)},Y_{n}(s)-Y_{m}(s)}_{\Hmo}\\&
\quad +\Re\scal{(\Pi_n-\Pi_m)\pp{\tilde{f}\pr{s,Y_{n}(s)+\Pi_nW_g(s)}-\tilde{f}\pr{s,Y_{n}(s)+W_g(s)}},Y_{n}(s)-Y_{m}(s)}_{\Hmo}\\&\quad +\Re\scal{\Pi_m\pp{\tilde{f}\pr{s,Y_{n}(s)+\Pi_nW_g(s)}-\tilde{f}\pr{s,Y_{m}(s)+\Pi_mW_g(s)}},Y_{n}(s)-Y_{m}(s)}_{\Hmo}\\
&\leq -\Re\scal{(\Pi_n-I)\tilde{f}\pr{s,Y_{n}(s)+W_g(s)},Y_{m}(s)}_{\Hmo}\\&\quad +C\pr{\norm{Y_{n}(s)-Y_{m}(s)}_{\Hmo}^2+ \norm{W_g(s)-\Pi_mW_g(s)}_{\Hmo}^2+\norm{W_g(s)-\Pi_nW_g(s)}_{\Hmo}^2}.
\end{align*}
\item For the term containing $\beta$, we use a similar argument combined with the $\pp{\beta}_1$-Lipschitz property of $\beta$, to get
\begin{align*}
&\Re\scal{(\Pi_m\beta\pr{Y_{n}(s)+\Pi_mW_g(s)}-\Pi_n\beta\pr{Y_{n}(s)+\Pi_nW_g(s)},Y_{n}(s)-Y_{m}(s)}_{\Lo{2}}\\
&\leq \Re\scal{(\Pi_m\beta\pr{Y_{n}(s)+W_g(s)}-\Pi_n\beta\pr{Y_{n}(s)+W_g(s)},Y_{n}(s)-Y_{m}(s)}_{\Lo{2}}\\
&\quad +\pp{\beta}_1\pr{\norm{W_g(s)-\Pi_nW_g(s)}_{\Lo{2}}+\norm{W_g(s)-\Pi_mW_g(s)}_{\Lo{2}}}\norm{Y_n(s)-Y_m(s)}_{\Lo{2}}\\
&\leq -\Re\scal{(I-\Pi_n)\beta\pr{Y_{n}(s)+W_g(s)},Y_{m}(s)}_{\Lo{2}}\\
&\quad +C\pr{\norm{W_g(s)-\Pi_nW_g(s)}_{\Lo{2}}^2+\norm{W_g(s)-\Pi_mW_g(s)}_{\Lo{2}}^2}+\frac{1}{2}\norm{Y_n(s)-Y_m(s)}_{\Lo{2}}^2\\
\end{align*}
\end{enumerate}
Combining these estimates yields 
\begin{align*}
&\norm{Y_{n}(t)-Y_{m}(t)}_{\Hmo}^2\\&+2\Re\int_0^t\scal{\beta\pr{Y_{n}(s)+\Pi_mW_g(s)}-\beta\pr{Y_{m}(s)+\Pi_mW_g(s)},Y_{n}(s)-Y_{m}(s)}^2_{\Lo{2}} ds\\\leq &
C\int_0^t\norm{Y_{n}(s)-Y_{m}(s)}_{\Hmo}^2ds+\int_0^t\norm{Y_{n}(s)-Y_{m}(s)}_{\Lo{2}}^2\\
&\quad -2\Re\int_0^t\scal{(\Pi_n-I)\tilde{f}\pr{s,Y_{n}(s)+W_g(s)},Y_{m}(s)}_{\Hmo}\ ds\\
&\quad -2\Re\int_0^t\scal{(I-\Pi_n)\beta\pr{Y_{n}(s)+W_g(s)},Y_{m}(s)}_{\Lo{2}}\ ds\\&\quad +C\int_0^t\pr{\norm{W_g(s)-\Pi_nW_g(s)}_{\Lo{2}}^2+\norm{W_g(s)-\Pi_mW_g(s)}_{\Lo{2}}^2}\ ds+\norm{(\Pi_n-\Pi_m)X_0}_{\Hmo}^2.
\end{align*}
Using Gronwall's inequality, 
\begin{equation}\label{Estim3.0}\begin{split}
&\norm{Y_{n}(t)-Y_{m}(t)}_{\Hmo}^2+\int_0^t\norm{Y_{n}(s)-Y_{m}(s)}^2_{\Lo{2}} ds\\&\leq \int_0^T\mathbf{1}_{s\leq t}e^{Ct}\alpha_{n,m}(s)ds\\&+(C+1)e^{CT}\Big\{\int_0^T\pr{\norm{W_g(s)-\Pi_nW_g(s)}_{\Lo{2}}^2+\norm{W_g(s)-\Pi_mW_g(s)}_{\Lo{2}}^2}\ ds\\&\quad\quad\quad\quad+\norm{(\Pi_n-\Pi_m)X_0}_{\Hmo}^2\Big\},
\end{split}
\end{equation}where 
\begin{equation}\label{Alpha}
\begin{cases}
\alpha_{n,m}(s)=&-2e^{-Cs}\Re\scal{(1-\Pi_n)\beta\pr{Y_{n}(s)+W_g(s)},Y_{m}(s)}_{\Lo{2}}\\
&-2e^{-Cs}\Re\scal{(\Pi_n-I)\tilde{f}\pr{s,Y_{n}(s)+W_g(s)},Y_{m}(s)}_{\Hmo};\\
\alpha_{n}(s)=&-2e^{-Cs}\Re\scal{(1-\Pi_n)\beta\pr{Y_{n}(s)+W_g(s)},Y(s)}_{\Lo{2}}\\
&-2e^{-Cs}\Re\scal{(\Pi_n-I)\tilde{f}\pr{s,Y_{n}(s)+W_g(s)},Y(s)}_{\Hmo};\\
\end{cases}
\end{equation} 
The remaining term on the right hand-side is denoted by $\eta_{n,m}$ and satisfies $\lim_{m\rightarrow\infty}\eta_{n,m}=\eta_n$ and $\lim_{n\rightarrow\infty}\eta_n=0$.
First, let us take $t=T$ in \eqref{Estim3.0} and note that, for $n$ fixed,  by the weak convergence of $Y_{m}$ in $\mathbb{L}^2\pr{\pp{0,T};\Lo{2}}$ and in $\mathbb{L}^2\pr{\pp{0,T};\Hmo}$, 
\[\lim_{m\rightarrow\infty} \int_0^Te^{CT}\alpha_{n,m}(s)ds=\int_0^Te^{CT}\alpha_{n}(s)ds+\eta_n.\]
Fatou's Lemma applied to \eqref{Estim3.0} with $t=T$ yields
\begin{equation*}
\int_0^T\norm{Y_{n}(s)-Y(s)}^2_{\Lo{2}} ds\leq \int_0^T\mathbf{1}_{s\leq t}e^{Ct}\alpha_{n}(s)ds+\eta_n.
\end{equation*}
Let us now specify that \begin{align*}
&\abs{e^{-Cs}\Re\scal{(\Pi_n-I)\tilde{f}\pr{s,Y_{n}(s)+W_g(s)},Y(s)}_{\Hmo}}\\=&\abs{e^{-Cs}\Re\scal{\tilde{f}\pr{s,Y_{n}(s)+W_g(s)},(\Pi_n-I)Y(s)}_{\Hmo}}\\\leq &e^{-Cs}\norm{\tilde{f}\pr{s,Y_{n}(s)+W_g(s)}}_{\Hmo}\norm{(\Pi_n-I)Y(s)}_{\Lo{2}}.
\end{align*}
Owing to \eqref{Estim1} (iii) and to the fact that $(1-\Pi_n)Y$ converges strongly to $0$ in $\mathbb{L}^2\pr{\pp{0,T};\Lo{2}}$, combined with the previous remark allowing to deal with the term involving $\Hmo$, one deduces that $\int_0^T\mathbf{1}_{s\leq t}e^{Ct}\alpha_{n}(s)ds$ converges to $0$ as $n\rightarrow\infty$, which shows that \[Y_{n}\textnormal{ converges strongly in }\mathbb{L}^2\pr{\pp{0,T};\Lo{2}}\textnormal{ to }Y. \] 
As a consequence of the Lipschitz-property of $\beta$,  $Z=\beta\pr{Y+W_g}$ and, as a consequence of the Lipschitz property of $\tilde{f}$, $F=f(\cdot,Y(\cdot)+W_g(\cdot),\mu(\cdot))$. {For our readers' sake, we emphasize that this equally implies that $-\Delta Z$ can be identified with $-\Delta\beta\pr{X}$ as elements in $\mathbb{L}^2\pr{\pp{0,T};\pr{\Lo{2}}^*}$.}

Furthermore,  it follows that $\alpha_{n,m}$ converges in $\mathbb{L}^1\pr{\pp{0,T};\mathbb{R}}$ to $0$ implying that $Y$ is also the $C\pr{\pp{0,T};\Hmo}$-limit of $Y_{n}$ (as $n\rightarrow\infty$). {By writing down the limiting integral equality with respect to $\pr{\Lo{2}}^*$, one has the last condition in Definition \ref{DefSol}. 

It then follows that $$\int_0^t\Delta\beta(Y(s)+W_g(s))ds\in \pr{H_0^1}^*\pr{\mathcal{O}},$$Lebesgue-almost surely on $\pp{0,T}$, which in turn implies the second condition in Definition \ref{DefSol}.}

The uniqueness of the solution is obtained directly from Gronwall's inequality, by using the monotonicity and the Lipschitz property of the function $f$.\\
Finally, let us mention that \eqref{Estim1} implies that $X=Y+W_g$ belongs to $$\mathbb{L}^2\pr{\Omega;C\pr{\pp{0,T};\Hmo}}\cap\mathbb{L}^2\pr{\Omega\times \pp{0,T};\Lo{2}}.$$
It is then clear that $\mathbb{P}_X(t)$ is a probability measure belonging to the Wasserstein space $\mathcal{P}_2\pr{\Hmo}$ and whose support belongs to $\mathbb{L}^2\pr{\pp{0,T};\Lo{2}}$.

For the next step, let $Y_1$ and $Y_2$ be two solutions associated to $\mu$ and $\nu$, and by $X_j=Y_j+W_g$ for $j\in\set{1,2}$. Then, employing, as we have already done before, the differential formula, with $c=1+4\max\set{\pp{f}_1,2}$, 
\begin{align*}
e^{-ct}\norm{Y_1(t)-Y_2(t)}_{\Hmo}^2 ds\leq &(2\pp{f}_1-c)\int_0^te^{-cs}\norm{Y_1(s)-Y_2(s)}_{\Hmo}^2\ ds\\&+2\pp{f}_1\int_0^te^{-cs}W_2\pr{\mu(s),\nu(s)}^2\ ds.
\end{align*}
It follows that
\begin{align*}
\int_0^te^{-cs}\norm{Y_1(s)-Y_2(s)}_{\Hmo}^2\ ds\leq\frac{2\pp{f}_1}{c-2\pp{f}_1}\int_0^te^{-cs}W_2\pr{\mu(s),\nu(s)}^2\ ds.
\end{align*}
This inequality is provided $\omega$-wise for almost all $\omega\in \Omega$ and, by taking expectations, it follows that 
\begin{align*}
d_{\mathbb{L}^2\pr{\pp{0,T},e^{-ct}dt;\mathcal{P}_2\pr{\Hmo}}}\pr{\mathbb{P}_{Y_1(\cdot)}\mathbb{P}_{Y_2(\cdot)}}\leq \frac{2\pp{f}_1}{c-2\pp{f}_1}d_{\mathbb{L}^2\pr{\pp{0,T}, e^{-ct}dt;\mathcal{P}_2\pr{\Hmo}}}\pr{\mu(\cdot),\nu(\cdot)},
\end{align*} 
providing a contraction with the choice of $c$ specified before. Our result is now complete due to classical fixed point arguments.
\end{proof}

\begin{remark}\label{RemVarSol}
Note that when $g=0$ and $f$ only depends on the solution, but not on the law, the previous solution can be seen as strong in $\pr{H_0^1}^*\pr{\mathcal{O}}$  and therefore it can be considered also in the form 
$$X(t)=X_0+\Delta\int_0^t\beta\pr{X(s)}ds + \int_0^t f\pr{X(s)}ds, \  \  \   \forall t\in\pp{0,T},$$
which corresponds to a variational formulation in $\pr{H_0^1}^*\pr{\mathcal{O}}$.  
From the monotonicity of the operator $\beta$ we can directly obtain uniqueness also for the variational formulation.  \\
This form will appear in the following optimal control formulation, since the necessary Fitzpatrick function characterizes, in our complex context,  functions $\beta$ which are merely monotone, while the strong formulation, even for real cases, usually holds only in the strictly monotone framework.
\end{remark}

\section{Control Interpretation}
In this section, we set aside the noise perturbation, i.e., we consider $g=0$ and the drift is no longer required to have a mean-field dependency. As the readers surely understand, it is crucial in the previous proof to require the strict monotonicity of $\beta$. In the real framework, this requirement can be circumvented essentially through two methods. 
\begin{enumerate}
\item On one hand, monotone Lipschitz (real) $\beta$ provide strictly monotone $\beta^{-1}$. This fails to hold in the complex case as seen by considering $\beta(z)=iz$, for which $\beta^{-1}(z)=-iz$ is monotone, but fails to be strictly monotone.
\item On the other hand, in the real case, subdifferential representations can be employed and these allow generalizations to monotone, albeit not strictly monotone, $\beta$. 
\end{enumerate}
This section concentrates on this latter method in order to show that as such it does not hold representation results in the complex setting, but generalizations can be considered. It is for this reason that we choose to concentrate on deterministic systems and set aside the terms arising from stochastic considerations.
\subsection{Some Elements of Representation for Maximal Monotonic Operators}

We consider a real Banach space $\mathbb{V}$ and we recall that whenever $\Phi:\mathbb{V}\longrightarrow \mathbb{R}\cup\set{+\infty}$ is convex and lower semicontinuous (l.s.c.),  then
\[\partial \Phi(x):=\set{p\in \mathbb{V}^*\ :\ \Phi(y)\geq \Phi(x)+\scal{p,y-x}_{\mathbb{V}^*,\mathbb{V}},\ \forall y\in \mathbb{V}},\]denotes the \emph{subdifferential} of $\Phi$ at $x\in \mathbb{V}$ .\\

In the real case, the porous media equation can be seen as a control problem using the fact that $\beta$, when strictly monotone (hence $-\Delta \beta$ maximal monotone on $\pr{\mathbb{L}^2\pr{\mathcal{O};\mathbb{R}},\pp{\mathbb{L}^2\pr{\mathcal{O};\mathbb{R}}}^*}$) can be identified with the subdifferential of some $\Phi$ and $\Phi(x)+\Phi^*(p)=\scal{p,x}_{\mathbb{R}}$ only when $p\in\partial \Phi(x)$. \\

As for the definition of monotonicity,  the one of subdifferentials can be extended to complex Banach spaces by following the spirit of the Remark \ref{RemFullRangeandRiesz}.

\begin{definition}
If $\mathbb{V}$ is a complex Banach space and $\Phi:\mathbb{V}\longrightarrow \mathbb{R}\cup\set{+\infty}$ is convex and lower semicontinuous, \[\partial \Phi(x):=\set{p\in \mathbb{V}^*\ :\ \Phi(y)\geq \Phi(x)+\Re\scal{p,y-x}_{\mathbb{V}^*,\mathbb{V}},\ \forall y\in \mathbb{V}},\]denotes the \emph{subdifferential} of $\Phi$ at $x\in \mathbb{V}$.
\end{definition}

\begin{example}Now, the simplest example one has in mind is the Schrödinger operator corresponding, up to a non-negative constant to $z\mapsto iz$. 
In this framework, if we want to represent $\beta=\partial\Phi$ as given by the previous definitions, this leads to \begin{align*}
\Phi(y)- \Phi(x)\geq \Re\scal{ix,y-x}.
\end{align*}
Taking (the upper limit as) $x\rightarrow 0$ leads to $\Phi(y)\geq \limsup_{z\rightarrow 0}\Phi(z)$. Let us consider $y:=rx$, with $r\in\mathbb{R}$. The previous inequality reads \[\Phi(x)\leq \Phi(rx)-\abs{x}^2\Re\scal{i,r-1}=\Phi(rx),\]
and taking $r\rightarrow 0$ leads to $\limsup_{z\rightarrow 0}\Phi(z)\geq \Phi(x)$ which implies $\Phi$ is constant and this leads to a contradiction. It follows that even for the simplest cases one has to find a cleverer way to "represent" $\beta$.
\end{example}

In the complex case,  given a monotone $\beta$, we define,  inspired by \cite{Fitz_1988} (see also \cite{SZ_2004}),
\begin{equation}\label{Fitz}
\mathcal{F}_\beta(z_1,z_2):=\Re{\scal{z_1,z_2}_\mathbb{C}}-\inf_{u\in \mathbb{C}}\Re{\scal{z_1-u,z_2-\beta(u)}_\mathbb{C}}
\end{equation}

We shall use in the control interpretation of the problem the following properties of the Fitzpatrick function.

\begin{proposition}\label{PropFitzpatrick}Let $\beta$ be a monotone function. The following properties hold true.
\begin{enumerate}
\item $\mathcal{F}_\beta:\mathbb{C}^2\longrightarrow\mathbb{R}\cup\set{+\infty}$ is a proper lower semi-continuous convex function.
\item For all $z_1,z_2\in\mathbb{C}$,  $\mathcal{F}_\beta(z_1,z_2)\geq\Re\scal{z_1,z_2}.$
\item For all $z_1\in\mathbb{C}$, $\beta(z_1)$ is the unique $z_2\in \mathbb
{C}$ such that $\mathcal{F}_\beta(z_1,z_2)=\Re\scal{z_1,z_2}_\mathbb{C}$.
\end{enumerate}
\end{proposition}
\begin{proof}[Proof of Proposition \ref{PropFitzpatrick}]
The reader is invited to note that \[\mathcal{F}_\beta(z_1,z_2)=\sup_{u\in\mathbb{C}}\Re \pp{\scal{\begin{pmatrix}
z_1\\z_2
\end{pmatrix},\begin{pmatrix}
\beta(u)\\u
\end{pmatrix}}_{\mathbb{C}^2}-\scal{u,\beta(u)}_{\mathbb{C}}},\](i.e.  a supremum over a family of linear functions), thus providing a real-valued lower semi-continuous convex function on $\mathbb{C}^2$. \\
Since $\beta$ is continuous and monotone,  $\begin{pmatrix}
\Re\beta\\\Im\beta
\end{pmatrix}:\mathbb{R}^2\longrightarrow\mathbb{R}^2$ is a maximal monotone operator, where, by abuse of notation, we have identified $\Re \beta(z)=\Re \beta\pr{\begin{pmatrix}
\Re z\\\Im z
\end{pmatrix}}$ when $z\in\mathbb{C}$ (and similarly for $\Im\beta$). By definition, $\mathcal{F}_\beta$ can be identified with the function $g:\mathbb{R}^2\times\mathbb{R}^2\longrightarrow\mathbb{R}\cup\set{+\infty}$ given by
\begin{align*}
g\pr{\Re z_1,\Im z_1,\Re z_2,\Im z_2}:=\scal{\begin{pmatrix}
\Re z_1\\\Im z_1
\end{pmatrix},\begin{pmatrix}
\Re z_2\\\Im z_2
\end{pmatrix}}_\mathbb{C}-inf_{u\in\mathbb{R}^2}\scal{\begin{pmatrix}
\Re z_1\\\Im z_1
\end{pmatrix}-u,\begin{pmatrix}
\Re z_2\\\Im z_2
\end{pmatrix}-\begin{pmatrix}
\Re\beta\\\Im\beta
\end{pmatrix}(u)
}_\mathbb{C}.
\end{align*}\\
This corresponds to the Fitzpatrick function associated to $\begin{pmatrix}
\Re\beta\\\Im\beta
\end{pmatrix}$ (see \cite[Eq. (1.2)]{SZ_2004}) and the assertions are merely re-interpretations of \cite[Eq. (1.3)]{SZ_2004} for the case of Hilbert spaces $\mathbb{R}^2$ (where the dual is identified with $\mathbb{R}^2$).
\end{proof}

In the spirit of the Brezis-Ekeland variational principle,  we can construct the following two optimal control problem which are equivalent to the existence result.

\subsection{A Variational Formulation}

In connection with the equation \eqref{ecc}, we formally define the following $u$-controlled dynamics
\begin{equation}\label{ecu}
dX(t)=\pr{\Delta u(t)+f(X(t))}dt.
\end{equation}
With respect to the aforementioned dynamics, we consider the control functional\begin{equation}
\label{Jfunc}
J(x,u):=\int_0^T\int_{\mathcal{O}} \pr{\mathcal{F}_\beta\pr{X^{x,u}(t),u(t)}-\Re\scal{X^{x,u}(t),u(t)}}d\xi dt.
\end{equation}
In the notion of solution, one seeks an integral expression of type \[X^{x,u}(t)=x+\Delta\int_0^t u(s)ds+\int_0^tf\pr{X^{x,u}(s)}ds,\]{with $u$ taking its values in $H_0^1$ and such that $\int_0^\cdot u(s)ds\in\mathbb{L}^\infty\pr{\pp{0,T};H_0^1}$ (please take a look at Definition \ref{DefSol}).}
The reader is invited to note the fact that, by the point 3 of the Proposition \ref{PropFitzpatrick},  an optimal pair $(X^*,u^*)$ of the previous problem which satisfies also $J(X^*,u^*)=0$ is a solution to \eqref{ecc}.

In order to ensure the well-posedness of the problem above,  we shall first write the following equivalent formulation 
\begin{equation}
\label{Jfunc}
J(X^{x,v},v):=\int_0^T\int_{\mathcal{O}} \pr{\mathcal{F}_\beta\pr{X^{x,v}(t),\partial_tv(t)}-\Re\scal{X^{x,v}(t),\partial_tv(t)}}d\xi dt,
\end{equation}
which is subject to 
\begin{equation}\label{eqv}X^{x,v}=x+\Delta v_t+\int_0^t f\pr{X^{x,v}(s)}ds,\end{equation}with a simple notation $v(t):=\int_0^t u(s)ds$. \\ 

\noindent The requirement on the control is now simplified, and amounts to $v\in \mathbb{L}^\infty\pr{\pp{0,T};H_0^1}$ which, impacts the solution $X^{x,u}$ with the requirement that $X^{x,u}\in\mathbb{L}^\infty\pr{\pp{0,T};\pr{H_0^1}^*}$. The Lipschitz requirement on $f$ with respect to $\pr{H_0^1}^*$ takes care of the remaining integral term in \eqref{eqv}. Furthermore, by a slight abuse of notation and in preparation of the precise statement of our problem, we no longer have a functional of the initial condition $x$, but rather of an element in $\mathbb{L}^\infty\pr{\pp{0,T};\pr{H_0^1}^*}$. \\

\noindent However, we still have to deal with the scalar product taken in $\Lo{2}$ and this amounts to imposing that $X^{x,v}\in \mathbb{L}^2\pr{\pp{0,T};\Lo{2}}$ respectively $v\in W^{1,2}\pr{\pp{0,T};\Lo{2}}$. \\
This equally gives us the consistency of the term involving Fitzpatrick's functional under the assumptions \ref{Ass1}. Indeed, since $\beta$ is assumed to be strictly monotone, 0 at 0, and Lipschitz,-continuous, it follows that  \[\abs{\beta(u)}\leq C\abs{u},\textnormal{ and }\mathcal{R}\scal{u,\beta(u)}\geq c\abs{u}^2, \] for some positive real constants $c,C$. This leads to \begin{align*}
\mathcal{R}\pr{\scal{u,z_2}+\scal{\beta(u),z_1}-\scal{u,\beta(u)}}&\leq \frac{c}{2}\abs{u}^2+\frac{1}{2c}\abs{z_2}^2+\frac{c}{2}\abs{u}^2+\frac{C^2}{2c}\abs{z_1}^2-c\abs{u}^2\\&=\frac{1}{2c}\abs{z_2}^2+\frac{C^2}{2c}\abs{z_1}^2.
\end{align*}As a consequence, $\mathcal{F}_\beta\pr{X^{x,u}(t),\partial_tv(t)}\in \mathbb{L}^1\pr{\pp{0,T}\times \mathcal{O};\mathbb{R}}$.\\
Finally, invoking the Lipschitz regularity of $f$, this time with respect to $\Lo{2}$, gives us the consistency in \eqref{Jfunc} in $\Lo{2}$.\\

\noindent As a consequence, one can concentrate on the following.

\begin{problem}\label{CtrlProb1}
\begin{align*}
\textnormal{Minimize }&J(y,v):=\int_0^T\int_{\mathcal{O}}\mathcal{F}_\beta\pr{y(t),\partial_tv(t)}+\Re\scal{v(t),f(y(t))}d\xi dt\\&+\frac{1}{2}\norm{v(T)}^2_{\Ho{}}-\Re\int_{\mathcal{O}}\scal{v(T),y_0+\int_0^Tf(y(s))ds}d\xi,\\
{\textnormal{ over }}&{y\in \mathbb{L}^2\pr{\pp{0,T};\Lo{2}}\cap \mathbb{L}^\infty\pr{\pp{0,T};\pr{H_0^1}^*},}\\ &{v\in W^{1,2}\pr{\pp{0,T};\Lo{2}}\cap \mathbb{L}^\infty\pr{\pp{0,T};H_0^1},}\\
\textnormal{subject to }&\Delta v(t)=y(t)-y_0-\int_0^tf(y(s))ds.
\end{align*}
\end{problem}
\begin{remark}To understand this, the reader is invited to note that if $y$ satisfies the state constraint, then, with standard integration by parts, one gets \begin{align*}
&\int_0^T\int_{\mathcal{O}}\Re\scal{\partial_tv(s),y(t)}d\xi dt\\&=\int_0^T\int_{\mathcal{O}}\Re\scal{\partial_tv(t),y_0+\int_0^tf(y(r))dr}d\xi dt+\int_{\mathcal{O}}\Re\scal{\partial_tv(t),\Delta v(t)}d\xi dt\\
&=\Re\int_{\mathcal{O}}\pp{\scal{v(T),y_0+\int_0^Tf(y(s))ds}-\int_0^T\scal{v(t),f(y(t))}dt}d\xi-\frac{1}{2}\norm{\nabla v(T)}^2_{\Lo{2}},
\end{align*}such that the cost functional is exactly
\[J(y,v):=\int_0^T\int_{\mathcal{O}}\mathcal{F}_\beta\pr{y(t),\partial_tv(t)}-\Re\scal{\partial_tv(t),y(t)}d\xi dt.\]
\end{remark}
{
As such, $J\geq 0$ (according to the second assertion in Proposition \ref{PropFitzpatrick}) and the $0$ value is attained for $\partial_tv(t)=\beta(y(t))$, that is for $y$ being the solution $X^{y_0,v}$ (which follows from the third assertion in Proposition \ref{PropFitzpatrick}). These consideration put together and by invoking Theorem \eqref{ThMain} give the following Brézis-Ekeland characterization.}

{\begin{proposition}\label{PropBE}
We ask Assumption \ref{Ass1} to hold true. Then, for every $y_0\in \Lo{2}$, the problem \ref{CtrlProb1} has a unique optimal solution $(y^*,v^*)$ such that \begin{enumerate}
\item[(i)] $y^*\in \mathbb{L}^2\pr{\pp{0,T};\Lo{2}}\cap \mathbb{L}^\infty\pr{\pp{0,T};\pr{H_0^1}^*}$;
\item[(ii)] $v^*\in W^{1,2}\pr{\pp{0,T};\Lo{2}}\cap \mathbb{L}^\infty\pr{\pp{0,T};H_0^1}$, and
\item[(iii)] $\Delta v^*(t)=y^*(t)-y_0-\int_0^tf(y^*(s))ds$ almost everywhere.
\end{enumerate}
Furthermore, \begin{enumerate}
\item[(iv)] $J(y^*,v^*)=0$ and
\item[(v)] $\partial_tv^*=\beta(y^*)$, and $y^*=X^{y_0,\beta(y^*)}$ is the unique solution to \eqref{ecc}\footnote{\color{black} Please note that $g=0$ and, in this case, \eqref{ecc} is actually a  deterministic equation.} starting from $X_0=y_0$.
\end{enumerate}
\end{proposition}}
We end this subsection with some remarks that motivate the consideration of the control problem, besides the generalization of known variational principles.
\begin{remark}\label{RemBE}
\begin{enumerate}
\item The above formulation is also valid for the classic Schrödinger operator and, more general, for merely monotone $\beta$. Indeed, the solvability of the control problem \ref{CtrlProb1} together with a null optimal value as in point (iv) implies the existence of a solution to \eqref{ecc} in the distributional sense as pointed out in Remark \ref{RemVarSol}.
\item The Problem \ref{CtrlProb1} is classically formulated (semi-continuity, convexity, well-posedness of terms), the only aspect missing being the \emph{coercitivity}. However, it is known that Fitzpatrick's function $\mathcal{F}_\beta$ is not the unique one \emph{representing} the maximal monotone function $\beta$ (see, for instance \cite{Visintin2013}). This approach can equally provide a solution to special cases when $\beta$ is monotonic but not strictly monotonic; we recall that its maximal monotonicity is still guaranteed in $\mathbb{C}$.
\item Again the reference \cite[Section 5]{Visintin2013} formally deals with stability of equations and this can be extended to our porous-media complex setting.  However,  the classical regularizations of $\beta$ to guarantee strict monotonicity i.e. $\beta_\varepsilon:=\beta+\varepsilon \Re Id$ followed by the natural Fitzpatrick choice of representatives $\mathcal{F}_{\beta_{\varepsilon}}$ is not covered by the abstract assumptions related to $\Gamma$-convergence given as examples in \cite[Section 5]{Visintin2013}. This is not entirely surprising, partly due to the previous remark.
\item Again under the assumption of strict monotonicity on $\beta$, another choice of maximal operator (see Proposition \ref{PropMaxmonotone}) is $-A=-\Delta\beta$ on $\pr{V,V^*}:=\pr{\Lo{2},\pr{\Lo{2}}^*}$ (see the aforementioned result and discussions preceding it). In this case too, Fitzpatrick's function $\mathcal{F}_{-A}$ can be defined by setting \[\mathcal{F}_{-A}(x,x^*):=\sup_{(a,a^*)\in gr (-A)}\Re\pp{\scal{a,x^*}_{\pr{V,V^*}}+\scal{x,a^*,}_{\pr{V,V^*}}-\scal{a,a^*}_{\pr{V,V^*}}},\]where $gr$ denotes the graph (see \cite[Eq. (1.2)]{SZ_2004}). The properties are similar to the ones exhibited in Proposition \ref{PropFitzpatrick} and another variational problem equivalent to the consistency of \eqref{ecc} can be formulated.
\end{enumerate}
\end{remark}

\subsection*{Submission statement}
\noindent The work presented here has not been published previously,  it is not under consideration for publication elsewhere.\\
The publication is approved by all authors and by the responsible authorities where the work was carried out. If accepted, it will not be published elsewhere in the same form, in English or in any other language, including electronically without the written consent of the copyright-holder.
\subsection*{Declaration of interest} 
The authors have no competing interest to declare.
\subsection*{Declaration of generative AI in scientific writing}
The paper makes no use of generative AI.
\subsection*{Author contributions}
\textbf{Ioana Ciotir}: Formal analysis; Funding acquisition; Investigation; Methodology; Supervision; Writing-original draft;\\
\textbf{Dan Goreac}: Formal analysis; Funding acquisition; Investigation; Methodology; Supervision; Writing-original draft;\\
\textbf{Juan Li}: Formal analysis; Funding acquisition; Investigation; Methodology; Supervision; Writing-original draft;\\
\textbf{Xinru Zhang}: Formal analysis; Investigation; Methodology; Writing-original draft.

\subsection*{Data Availability Statement}
\noindent No new data were created or analyzed in this study. Data sharing is not applicable to this article.
\subsection*{Acknowledgment}
Part of this work was completed while Ioana Ciotir occupied a visiting position at Shandong University Weihai; she wishes to thank the School of Mathematics and Statistics for their hospitality.\\
Dan Goreac, Juan Li and Xinru Zhang have been partially supported by the NSF of Shandong Province (NO. ZR202306020015), National Key R and D Program of China (NO. 2018YFA0703900), and the NSF of P.R. China (NO. 12031009).

\bibliographystyle{abbrv}
\bibliography{bibl_2024_fin}
\end{document}